\documentclass[11pt,A4paper]{article}

\usepackage[left=32mm,right=32mm,top=30mm,bottom=32mm]{geometry}
\usepackage{labelfig}
\usepackage{epsfig}
\usepackage{epstopdf}

\usepackage{xfrac}

\usepackage[percent]{overpic}

\usepackage{color}
\usepackage{amsthm,amsmath,amssymb}
\usepackage{booktabs}
\usepackage{mathpazo}
\usepackage{microtype}
\usepackage{overpic}
\usepackage{bm}
\usepackage{sectsty}
\usepackage[
	pdftitle={PDFTitle},
	pdfauthor={Hugo Parlier},
	ocgcolorlinks,
	linkcolor=linkred,
	citecolor=linkred,
	urlcolor=linkblue]
{hyperref}

\definecolor{linkred}{RGB}{220,20,60}
\definecolor{linkblue}{RGB}{16, 78, 139}

\usepackage[hang,flushmargin]{footmisc}

\usepackage[shortlabels]{enumitem}

\usepackage{titlesec}
	\titlespacing{\section}{0pt}{12pt}{0pt}
	\titlespacing{\subsection}{0pt}{6pt}{0pt}
	
\titlelabel{\thetitle.\quad}

\makeatletter 

\long\def\@footnotetext#1{%
\H@@footnotetext{%
\ifHy@nesting 
\hyper@@anchor{\@currentHref}{#1}%
\else 
\Hy@raisedlink{\hyper@@anchor{\@currentHref}{\relax}}#1%
\fi 
}}

\def\@footnotemark{%
\leavevmode 
\ifhmode\edef\@x@sf{\the\spacefactor}\nobreak\fi 
\H@refstepcounter{Hfootnote}%
\hyper@makecurrent{Hfootnote}%
\hyper@linkstart{link}{\@currentHref}%
\@makefnmark 
\hyper@linkend 
\ifhmode\spacefactor\@x@sf\fi 
\relax 
}%

\ifFN@multiplefootnote%
\renewcommand*\@footnotemark{%
\leavevmode 
\ifhmode 
\edef\@x@sf{\the\spacefactor}%
\FN@mf@check 
\nobreak 
\fi 
\H@refstepcounter{Hfootnote}%
\hyper@makecurrent{Hfootnote}%
\hyper@linkstart{link}{\@currentHref}%
\@makefnmark 
\hyper@linkend 
\ifFN@pp@towrite 
\FN@pp@writetemp 
\FN@pp@towritefalse 
\fi 
\FN@mf@prepare 
\ifhmode\spacefactor\@x@sf\fi 
\relax%
}%
\fi 

\makeatother 

\theoremstyle{plain}
\newtheorem{theorem}{Theorem}[section]
\newtheorem{proposition}[theorem]{Proposition}
\newtheorem{lemma}[theorem]{Lemma}
\newtheorem{corollary}[theorem]{Corollary}

\theoremstyle{definition}
\newtheorem{definition}[theorem]{Definition}

\newtheorem{remark}[theorem]{Remark}

\newcommand{\R}{{\mathbb R}}

\newcommand{\CC}{{\mathcal C}}

\newcommand{\A}{{\mathcal A}}

\newcommand{\M}{{\mathcal M}}

\newcommand{\arcsinh}{{\,\rm arcsinh}}
\newcommand{\arccosh}{{\,\rm arccosh}}

\newcommand{\arctanh}{{\,\rm arctanh}}

\newcommand{\LL}{{\mathcal L}}

\sectionfont{\large \bfseries}
\subsectionfont{\normalsize}

\long\def\symbolfootnote[#1]#2{\begingroup%
\def\thefootnote{\fnsymbol{footnote}}\footnote[#1]{#2}\endgroup}

\def\blfootnote{\xdef\@thefnmark{}\@footnotetext}

\begin{document}

{\Large \bfseries Interrogating surface length spectra and quantifying isospectrality}

{\large Hugo Parlier\symbolfootnote[1]{\normalsize Research supported by Swiss National Science Foundation grant number PP00P2\textunderscore 153024 \\
{\em MSC2010:} Primary: 32G15, 30F10. Secondary: 30F60, 53C22, 58J50, 11F72. \\
{\em Key words and phrases:} spectral theory, closed geodesics on hyperbolic surfaces, isospectral families, moduli spaces}
 }

{\bf Abstract.} 
This article is about inverse spectral problems for hyperbolic surfaces and in particular how length spectra relate to the geometry of the underlying surface. A quantitative answer is given to the following: how many questions do you need to ask a length spectrum to determine it completely? In answering this, a quantitative upper bound is given on the number of isospectral but non-isometric surfaces of a given genus.
\vspace{1cm}

\section{Introduction}

This article is about inverse spectral problems for hyperbolic surfaces and in particular for how length spectra relate to the geometry of the underlying surface. The idea is to understand how much information about a surface $X$ is contained in the set of lengths (with multiplicities) of all of its closed geodesics $\Lambda(X)$.

The main goal here is to provide quantitive results which only depend on the topology of the underlying surface to several problems. Here is the first one:

{\it {\bf Problem 1:} How many {\it questions} must one ask a length spectrum to determine it completely?}

There are variations on this problem depending on what type of question one allows. Here we'll only allow one type of question. You're allowed to constitute a (finite) list of values you're not interested in and then ask for the smallest value not on the list. 

The approach taken to tackle this problem leads to another one:

{\it {\bf Problem 2:} How many non-isometric isospectral surfaces of genus $g$ can their be?}

This last problem is a well-studied problem in inverse spectral theory and for closed surfaces, via the Selberg trace formula, length isospectrality is equivalent to Laplace isospectrality. 

There are different known techniques to produce examples of non-isometric isospectral hyperbolic surfaces. That such surfaces exist might seem surprising, and in many ways they are an extremely rare phenomenon. McKean \cite{McKean} showed that at most a finite number of other surfaces can be isospectral but non-isometric to a given hyperbolic surface and Wolpert \cite{Wolpert} showed that outside of a certain proper real analytic subvariety of the moduli space $\M_g$ of genus $g\geq 2$ surfaces , all surfaces are uniquely determined by their length spectrum. The first examples were due to Vign\'eras \cite{Vigneras1,Vigneras2}. A multitude examples, not only in the context of surfaces, were found using a technique introduced by Sunada \cite{Sunada}, namely those of Gordon, Webb and Wolpert \cite{GordonWebbWolpert}, who also produced isospectral but non-isometric planar domains, answering a famous question of Kac \cite{Kac}. 

As one might expect, the size of sets of isospectral non-isometric surfaces (isospectral sets) can grow with the topology of the underlying surface. In particular Brooks, Gornet and Gustafson \cite{BrooksGornetGustafson} showed the existence of isospectral sets of cardinality that grows like $g^{c\log(g)}$ where $g$ is the genus and $c$ is an explicit constant. This improved previous results of Tse \cite{Tse}. It would be difficult to review all of the literature, but the point is that a lot of effort has gone into finding examples of isospectral sets, and in particular into producing lower bounds to Problem 2 above.

On the other end, in addition to theorems of McKean and Wolpert, the only quantitative upper bounds seem to be due to Buser \cite[Chapter 13]{BuserBook}. Buser shows that the cardinality of isospectral sets is bounded above by $e^{720g^2}$. The proof is based on a number of ingredients, one of them being a bound on lengths of geodesics in pants decompositions (so-called Bers constants, see \cite{Bers}). Since Buser's theorem, there have been improvements as to what is known about short pants decompositions which in turn lead to direct improvements of the $e^{720g^2}$ upper bound. However, a direct application of Buser's techniques will still give bounds on the order of $e^{Cg^2}$ for some constant $C$. To improve this bound significantly requires using something else and this is one of the goals of this paper. 

Concerning Problem 1, the main result is the following theorem.
\begin{theorem}\label{thm:mainquestion}
There exists an explicit universal constant $A$ such that the following holds. The length spectrum $\Lambda$ of a surface of genus $g$ can be determined by at most $g^{Ag}$ questions. 
\end{theorem}
The constant $A$ can be taken to be $154$, which is not in any way sharp. Note that in particular a finite number of lengths determine the full spectrum, a well known fact which was somehow at the origin of this project and the above result is one way of quantifying it.

We note there is another version of this finiteness, again due to by Buser \cite[Theorem 10.1.4]{BuserBook}, used to provide a simpler (or at least a different) proof of Wolpert's theorem mentioned previously. The result is that there is a constant $B_{g,\varepsilon}$ that only depends on the genus $g$ of the surface and a lower bound $\varepsilon$ on the systole such that the lengths of length less than $B_{g,\varepsilon}$ determine the full length spectrum. (The systole is the length of a shortest closed geodesic of the surface.) The proof is based on analyticity but there aren't any known quantifications of $B_{g,\varepsilon}$. Note that it must depend on $g$ and $\varepsilon$, unlike the quantification of Theorem \ref{thm:mainquestion}. It is also interesting to compare these results to other rigidity phenomena for length spectra, such as result of Bhagwat and Rajan \cite{BhagwatRajan}, which states that, for even dimensional compact hyperbolic manifolds, two length spectra are either equal or they differ by an infinite number of values. 

The proof of Theorem \ref{thm:mainquestion} involved determining possible isometry classes of surfaces with given length spectra. In particular, it is necessary to count the size of isospectral sets.

\begin{theorem}\label{thm:maincount}
There exists a explicit universal constant $B$ such that the following holds. Given $X\in M_g$ there are at most $
g^{Bg}$ surfaces in $\M_g$ isospectral to $X$.
\end{theorem}
As before, the constant $B$ can be taken to be $154$. Although this considerably lowers the upper bound, there is still a significant discrepancy between the lower and upper bounds ($g^{c\log(g)}$ vs. $g^{Cg}$). Finding the true order of growth seems like a challenging problem.

One of the main ingredients for obtaining these quantifications is to find a different parameter set for the moduli space of surfaces. Buser's approach uses Fenchel-Nielsen coordinates and pants decompositions. The problem is that even the shortest pants decompositions curves can get long (as least on the order of $\sqrt{g}$). In fact, even the shortest closed geodesic can get long too as there are families of surfaces whose systoles grow on the order of $\log(g)$. The parameter set proposed here is also based on a set of curves whose geometric data determine the surface. The set of curves has two different components: $\Gamma$ (the curves) and $\Gamma_\A$ (the chains). The set $\Gamma$ is a set of disjoint simple curves and $\Gamma_\A$ is a set of curves somehow attached to the elements of $\Gamma$ via a set of arcs $\A$ (see Section \ref{sec:cc} for a proper definition). The first main feature of a curve and chain system $\Gamma,\Gamma_\A$ is that lengths and twists of elements of $\Gamma$ and the lengths of elements of $\Gamma_\A$ determine a surface in moduli space (and in fact in Teichm\"uller space although that's not the point here). 

Their second main feature - and this is what truly distinguishes them from pants decompositions - is that we can bound their lengths by constants on the order of $\log(g)$.
\begin{theorem}\label{thm:cclength}
Any $X\in \M_g$ admits a curve and chain system $\Gamma, \Gamma_\A$ satisfying
$$
\ell(\gamma) < 2 \log(4g)
$$
for all $\gamma \in \Gamma$ and
$$
\ell(\gamma_a) < 8 \log(4g)
$$
for all $a \in \A$. 
\end{theorem}
This theorem is a type of Bers' constant theorem for curve and chain systems and the examples of surfaces with large systoles show that the order of growth is not improvable. The $\log(g)$ order of growth is an essential ingredient for obtaining the quantification in Theorem \ref{thm:maincount}. We note that Theorems \ref{thm:maincount}, \ref{thm:mainquestion} and \ref{thm:cclength} all remain true for complete finite area surfaces when replacing the dependence on $g$ with a dependence on the area but, in tune with most of the literature on the subject, we focus on closed surfaces.

Any approach to bounding the cardinality of isospectral sets requires more than just short curves. This is because when surfaces have very thin parts, lengths of curves that transverse these parts are long and necessary to identify the moduli of the thin parts. Identifying possible geometries of the thin part of surfaces requires another type of approach. Buser's approach to this is to treat "very" short curves differently from just "somewhat" short curves and in both cases provides involved, although mostly elementary, arguments. The approach taken here is quite different. 

First of all, we only differentiate between long and short curves where short means those of length less than $2 \arcsinh(1)$. This constant is a natural quantity when dealing with simple closed geodesics because is the exact value beyond which, on any complete hyperbolic surface, one can guarantee that any closed geodesic that intersects it is of greater length. To deal with a short curve, we first examine the topological types of the shortest curve transversal to it. Depending on their type, we employ two strategies. The first uses a recent theorem of Przytycki \cite{Przytycki} which bounds the number of arcs that pairwise intersect at most once. The second strategy, perhaps somewhat unexpectedly, uses McShane type identities \cite{McShane, Mirzakhani, TanWongZhang} as a measure on embedded pairs of pants. Estimates related to the lengths of curve and chain systems are then plugged in resulting on bounds on the size of isospectral sets.

{\bf Organization.} A preliminary section mostly contains standard results, sometimes adapted slightly for our purposes. This is followed by a full section on curve and chain systems culminating in the proof of Theorem \ref{thm:cclength}. Section \ref{sec:thin} is then dedicated to the different approaches used to deal with thin structures (although certain aspects concerning McShane identities are delayed to the appendix because they are of somewhat different nature to the rest). The results of the previous two sections are put together in Section \ref{sec:iso} to prove Theorem \ref{thm:maincount}. Finally, Section \ref{sec:question} is mainly the proof of Theorem \ref{thm:mainquestion}. 

{\bf Acknowledgements.}
Thanks are due to a number of people I have discussed aspects of this paper with including Ara Basmajian, Peter Buser, Chris Judge, Youngju Kim, Bram Petri and Binbin Xu. Some of the work on the paper was done during a visit to the Korean Institute for Advanced Study and I thank the institute members for their hospitality.

\section{Preliminaries}

The moduli space $\M_g$ is the space of complete hyperbolic structures on a closed orientable topological surface $\Sigma_g$ of genus $g\geq 2$ up to isometry. In order to describe $\M_g$, it is often useful to use lengths of closed geodesics. Given an element of $\xi \in\pi_1(\Sigma_g)$ and any $X\in \M_g$, there is a unique closed geodesic $\gamma_\xi\subset X$ in the free homotopy class of $\xi$. In that way, for every $\xi$, one gets a function $\ell_{\cdot}:\M_g \to \R^{+}$ where to any $X$ one attributes the value 
$$
\ell_X(\xi) := \ell(\gamma_\xi)
$$
This one-to-one correspondence between elements of the fundamental group and closed geodesics is very useful and in general we won't distinguish between a homotopy class and its geodesic realization. In fact, unless specifically stated, when the surface $X$ is clear from the context, $\ell(\gamma)$ will generally mean the length of the closed geodesic in the homotopy class of $\gamma$. 

A closed geodesic is primitive if is not the iterate of a another closed geodesic. A primitive closed geodesic is called simple if doesn't have any self-intersection points. We denote by
$$\Lambda(X):=\{\ell_1 \leq \ell_2 \leq \hdots \}$$
the ordered (but unmarked) set of lengths of primitive closed geodesics of $X$, counted with multiplicity. This means that if $X$ has $n$ closed geodesics of length $l$, $l$ will appear $n$ times. As $X$ is of finite type, the set $\Lambda(X)$ is always a discrete set. As we'll be trying to say things about a surface that has a given spectrum, we'll sometimes use $\Lambda$ by itself to denote a length spectrum without knowledge of the underlying surface. We note that everything we do could be done for the full length spectrum instead (thus including the lengths of non-primitive elements) as one is determined by the other, but we use primitive here for convenience. We say that $X$ and $Y$ are isospectral if $\Lambda(X) = \Lambda(Y)$. 

Cutting and pasting techniques will be used throughout the paper. For this, the following convention will be used. Given a simple closed geodesic $\alpha$ of a surface $X$, we can remove $\alpha$ from $X$. The result is an open surface (possibly disconnected) whose geometric closure has two boundary curves. We denote this closed manifold with boundary $X\setminus \alpha$. Similarly, we denote $X \setminus \mu$ for the corresponding geometric closure when $\mu$ is a geodesic multicurve.

The following result \cite{Keen}, see \cite[Thm. 4.1]{BuserBook} for this version, will be useful for our purposes.
\begin{lemma}[Collar lemma]\label{lem:collar}
Given such two disjoint simple closed geodesics $\gamma_1, \gamma_2 \subset X$, their collars
$$\CC(\gamma_i):= \{x\in X : d_X(x,\gamma_i)\leq w_i\}$$
of widths respectively 
$$
w_i:=\arcsinh\left( \frac{1}{\sinh(\sfrac{\ell(\gamma_i)}{2})}\right)
$$
for $i=1,2$ are embedded cylinders and disjoint.
\end{lemma}

We can compute the lengths of the boundary curves of the collar, which are {\it not} geodesic, as follows.
\begin{lemma}
The boundary lengths of the collar $\CC(\gamma)$ for a simple closed geodesic $\gamma$
are both equal to 
$$
\ell(\gamma)\coth\left(\frac{\ell(\gamma)}{2}\right)
$$
\end{lemma}

\begin{proof}
The length of the boundary of the set of points at distance $r$ from a curve of length $\ell$ in the hyperbolic plane is $\ell \cosh(r)$ so here, as everything is embedded in $X$, the length of both boundary curves are
$$
\ell (\gamma) \cosh\left( \arcsinh\left( \frac{1}{\sinh(\sfrac{\ell(\gamma)}{2})}\right) \right) = \ell(\gamma)\coth\left(\frac{\ell(\gamma)}{2}\right)
$$
as claimed.
\end{proof}

For a given surface $X$, the set of closed geodesics of length less than $2\arcsinh(1)$ will be denoted $\Gamma_0(X)$. Non-simple closed geodesics are all of length at least $2\arcsinh(1)$ (see \cite[Theorem 4.2.2]{BuserBook}) so $\Gamma_0(X)$ consists of simple closed geodesics. As a corollary of the collar lemma and of the previous lemma, we have the following about elements of $\Gamma_0(X)$. 

\begin{corollary}\label{cor:shortboundary}
Given distinct $\gamma_1,\gamma_2 \in \Gamma_0(X)$ their collars satisfy
$$\CC(\gamma_1)\cap \CC(\gamma_2) = \emptyset$$

Furthermore, if $\gamma \in \Gamma_0$, the boundary curves of its collar $\CC(\gamma)$ have length $b$ where
$$
2 \leq b \leq 2 \sqrt{2} \log(1+\sqrt{2})
$$
 \end{corollary}
\begin{proof}
The first statement follows directly from Lemma \ref{lem:collar}. The second statement follows from the monotonicity in $\ell$ of the function $\ell \coth\left(\frac{\ell}{2}\right)$ between $0$ and $2\arcsinh(1)$.
\end{proof}
In particular, geodesics in $\Gamma_0(X)$ are disjoint. Thus there are at most $3g-3$ closed geodesics of length at most $2\arcsinh(1)$. The other ones are more difficult to count, and the following lemma, based on a lemma of Buser \cite[Lemma 6.6.4]{BuserBook}, is a useful tool for this. Note the statement is not exactly the same but the proof contains all of the ingredients we'll need.

\begin{lemma}[Counting closed geodesics]\label{lem:countcurves}
Let $L>0$ and $X\in \M_g$. Then there are at most
$$
(g-1) \,e^{L+6}
$$
primitive closed geodesics of length most $L$.
\end{lemma}

\begin{proof}
The statement is not exactly the statement from \cite[Lemma 6.6.4]{BuserBook} but can be obtained by the proof as follows. 

There are at most $3g-3$ curves of length $2\arcsinh(1)$ and now we count the longer curves.

In the proof of \cite[Lemma 6.6.4]{BuserBook}, Buser shows that the number of closed geodesics that are {\it not} iterates of the curves of length less than $2\arcsinh(1)$ are bounded above by 
$$
\frac{\cosh(L+3r) - 1}{\cosh(r) -1} \frac{2(g-1)}{\cosh(\sfrac{r}{2})-1}
$$
where $r=\arcsinh(1)$. This bound is obtained by covering the thick part of the surface by balls of radius $\arcsinh(1)$ and by counting geodesic loops based in the centers of the balls above and using an area comparaison argument.

Now the total number of primitive curves is thus bounded by 
$$
3g-3 + \frac{\cosh(L+3r) - 1}{\cosh(r) -1} \frac{2(g-1)}{\cosh(\sfrac{r}{2})-1} = (g-1)\left( 3 + \frac{2\cosh(L+3r) - 2}{(\cosh(r) -1)(\cosh(\sfrac{r}{2})-1)}\right)
$$
and a straightforward calculus computation shows that this quantity is bounded above by
$$
(g-1)\,\, e^{L+6}
$$
for all $g>2$ and $L>0$ as desired.
\end{proof}

The above statement is in fact slightly weaker than the statement in \cite{BuserBook} but it will be exactly the statement we need as we're dealing with the primitive length spectrum. 

\begin{remark} The number of closed geodesics of length at most $L$ grows asymptotically like $\sfrac{e^L}{L}$, a result due to Huber \cite{Huber}, generalized to many different contexts by Margulis \cite{Margulis}. Although the bound in Lemma \ref{lem:countcurves} does not exhibit the correct order of growth, it has the advantage of being effective and working for {\it any} $L$. We also note that although the asymptotic growth of the number of simple closed geodesics is considerably slower (it is polynomial in $L$, a result of Mirzakhani \cite{MirzakhaniGrowth}), the above lemma, for sufficiently small values of $L$, provides an effective upper bound on the growth which is interesting even for the simple geodesics.
\end{remark}

Another ingredient will be the following result, due to Bavard \cite{Bavard}, which can be used to find short curves on surfaces. It will be a useful tool and so we state it as a lemma. The only thing we really need is an upper bound on the length of non-trivial curve that grows like $2 \log(g)$ which can be directly obtained by an area comparison argument. This result, remarkably, is sharp and so we state it precisely.
\begin{lemma}]\label{lem:systole} For any $X\in \M_g$ and any $x\in X$, there exists a geodesic loop $\delta_x$ based in $x$ such that 
$$
\ell(\delta_x) \leq 2 \arccosh \left(\frac{1}{2 \sin \frac{\pi}{12g-6}}\right) 
$$
\end{lemma}
For future reference we denote
$$
R_g:= \arccosh \left(\frac{1}{2 \sin \frac{\pi}{12g-6}}\right) 
$$
That such a bound exists is relatively straightforward: the area of a ball in the hyperbolic plane grows exponentially in its radius whereas the area of a surface of genus $g$ is $4\pi(g-1)$ thus the radius of an embedded ball cannot exceed $\log(g)$ by some large amount.

\section{Curve and chain systems}\label{sec:cc}

We'll begin by describing a decomposition of a surface obtained by cutting along simple closed geodesics and simple orthogeodesics between them. This will lead to what will be called an {\it curve and arc} decomposition. We'll then show how the arcs relate to a family of curves to obtain {\it curve and chain systems}. The lengths of these curves will help determine isometry types of surfaces. 

\subsection{Definition and topological types}

\begin{definition}
A curve and arc decomposition of $X \in \M_g$ is a non-empty collection of disjoint simple closed geodesics $\Gamma=\gamma_1,\hdots,\gamma_k$ and a collection of arcs $\A=a_1,\hdots,a_{6g-6}$ on $X \setminus \Gamma$ such that $X \setminus \{\Gamma, \A\}$ is a collection of geodesic right angled hexagons.
\end{definition}

We shall refer to the $\gamma_1,\hdots,\gamma_k$ as the curves and the $a_1,\hdots,a_{6g-6}$ as the arcs (or as orthogeodesics, as they are orthogonal to the curves in their endpoints). 

{\it Example.} A first example of curve and arc decomposition can be obtained by taking $\Gamma$ to be a pants decomposition of $X$ (a collection of disjoint simple closed geodesics that decompose $X$ into three holed spheres or pairs of pants). The set of arcs $\A$ is a collection of three disjoint simple orthogeodesics between boundary curves on each pair of pants.

Observe that, although the definition given is geometric, it could have been made purely topological as follows. The set $\Gamma$ is a set of disjoint non-isotopic essential simple closed curves and $\A$ is a maximal set of disjoint simple arcs on $X \setminus \Gamma$ with endpoints on $\Gamma$. Further note that any non-empty set $\Gamma$ of disjoint simple curves can be completed into a curve-arc decomposition by taking $\A$ to be any maximal set of arcs on $X \setminus \Gamma$ with endpoints on $\Gamma$. 

We'll need to count the number of different topological types of curve and arc decompositions that one can have in genus $g$. For future use, it will be useful to have a marking on the curves $\Gamma$ and arcs $\A$. This is simply a labelling of the curves and arcs and we'll refer to these as marked curve and arc decompositions. Two marked curve and arc decompositions are topologically equivalent if there is a homeomorphism between them which respects the marking on the curves and arcs. 

\begin{lemma}\label{lem:countca}
The number of different topological types of marked curve and arc decompositions is bounded above by 
$$
\frac{1}{e^6} \left( \frac{12^6}{e^5}\right)^{g-1} (g-1)^{6g-6}
$$
\end{lemma}

\begin{proof}
A curve and arc decomposition is a decomposition of $X$ into $4g-4$ hexagons. We see $X$ as the result of a two step construction. First we paste together the hexagons to obtain $X\setminus \Gamma$ and then we paste together the boundary curves of $X\setminus \Gamma$ to obtain $X$.

We begin with a set of $4g-4$ hexagons, each with a set of three marked non-adjacent oriented side arcs. For instance, we put the hexagons all in the plane and give the orientation to arcs induced by a fixed orientation of the plane. We'll get rid of the orientation on the arcs later, but it will be useful to retain it for the time being. The hexagons have a "frontside" and a "backside". We're going to paste the hexagons together to obtain an orientable two-sided surface where the front sides of each side hexagon are all on the same side. 

To do so take a first side arc and paste it to another. Note there is a unique way of pasting so that both front sides of the hexagon are on the same side. We retain, as orientation for the resulting arc, the orientation of the first arc.

In all there are $12g-12$ side arcs. That means that there are 
$$
(12g-12)!! = (12g-11) \cdot (12g-9) \cdots 6\cdot 3 = \frac{(12g-12)!}{2^{\sfrac{12g-12}{2}}(\sfrac{12g-12}{2})!}
$$
ways of doing this. Using standard inequalities on factorials gives at most 
$$
\frac{(12g-12)!}{2^{6g-6}(6g-6)!} \leq \frac{(12g-12)^{(12g-12+\sfrac{1}{2})}}{2^{6g-6}e^{12g-11}} \frac{e^{6g}}{\sqrt{2\pi} (6g-g)^{6g-6+\sfrac{1}{2}}}
$$
which after further simplifications becomes
$$
\frac{1}{e^6} \left( \frac{12^6}{e^5}\right)^{g-1} (g-1)^{6g-6}
$$
Now the result is a surface with boundary curves and marked oriented arcs. The marked oriented arcs mark the boundary curves, so in particular, we also have all necessary information to perform the second step of our process (the pasting of the curves $\Gamma$) without any further counting. As we don't care about orientations we can forget them. Certainly we've counted every topological type of marked curve and arc decomposition and the lemma is proved.
\end{proof}

The main purpose of this counting lemma will be to count curve and chain systems which will be introduced in the next section.

\subsection{Coordinates for moduli space}\label{sec:curveandchain}

The goal is to use geometric quantities to determine the isometry class of surfaces. These coordinates are somewhat similar to Fenchel-Nielsen coordinates (but, unlike the latter, some of ours will be redundant). 

Like for Fenchel-Nielsen coordinates it will be necessary to consider twist parameters, but only along curves of $\Gamma$. To do so we consider a way of marking points on each side of a curve of $\Gamma$. An example of how to do this is to choose, for every $\gamma\in \Gamma$ and each side of $\gamma$, an endpoint of an arc $a\in \A$ which has is attached to $\gamma$ and which "leaves" on the corresponding side. The specific marked point associated to a metric structure $X\in \M_g$ is obtained by taking the geodesic realizations of $\gamma$ and $a$ (where $a$ is now an orthogeodesic) and taking the appropriate intersection point between $a$ and $\gamma$ (there might be two of them). In this way, for every $\gamma \in \Gamma$ and every $X\in \M_g$, one has two marked points, say $p^+_\gamma$ and $p^-_\gamma$, one for each side of $\gamma$. 

\begin{definition} The twist parameter along $\gamma$ is the signed distance between $p^+_\gamma$ and $p^-_\gamma$.
\end{definition}

We denote by $\tau(\gamma)$ the twist parameter of $\gamma$. Similarly, $\tau(\Gamma)$ is the set of (marked) twist parameters of $\Gamma$. Similarly, $\ell(a)$ is the length of the unique orthogeodesic in the free homotopy class of $a$ and we'll denote by $\ell(\A)$ the  set of (marked) lengths of $\A$.

\begin{lemma}\label{lem:arc}
The parameters $\tau(\Gamma)$ and $\ell(\A)$ uniquely determine $X\in \M_g$. 
\end{lemma}

\begin{proof}
The length of the arcs of $\A$ determine the geometry of each of the hexagons. They in turn determine the lengths of $\Gamma$. The only thing remaining is how to determine how the elements of $\Gamma$ are pasted but this is determined by $\tau(\Gamma)$. 
\end{proof}

Recall our goal is to relate the length spectrum to isometry classes of surfaces so although the curve-arc length parameters are convenient, we'd like a set of parameters which only use curves. To do so we replace the lengths of arcs by the lengths of curves as follows. 

First given $a\in \A$, we define a free homotopy class (or equivalently a closed geodesic) as follows. Give $a$ and $X\setminus \Gamma$ an orientation (this defines an orientation on the end geodesics of $a$, say $\gamma_1$ and $\gamma_2$). Now let $\gamma_a$ be the closed geodesic in the free homotopy class of
$$\gamma_1 * a * \gamma_2  * a^{-1}$$
Note that if $a$ is arc between distinct curves $\gamma_1$ and $\gamma_2$, then $\gamma_a$ is simple. Otherwise it is a closed geodesic with two self-intersection points as in the right side of Figure \ref{fig:pants}. (Here distinct curves means distinct curves on $X\setminus \Gamma$: they could very well correspond to the same curve on $X$.)
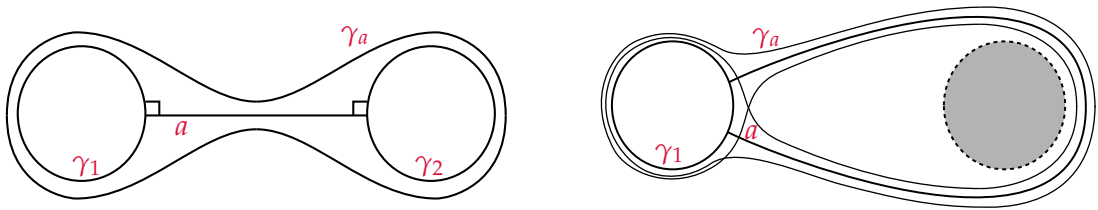
\begin{figure}[h]
{\color{linkred}
\leavevmode \SetLabels
\L(.175*.37) $a$\\%
\L(.32*.82) $\gamma_a$\\
\L(.385*.2) $\gamma_2$\\%
\L(.085*.2) $\gamma_1$\\
\L(.674*.35) $a$\\%
\L(.68*.81) $\gamma_a$\\
\L(.595*.26) $\gamma_1$\\
\endSetLabels
\begin{center}
\AffixLabels{\centerline{\epsfig{file =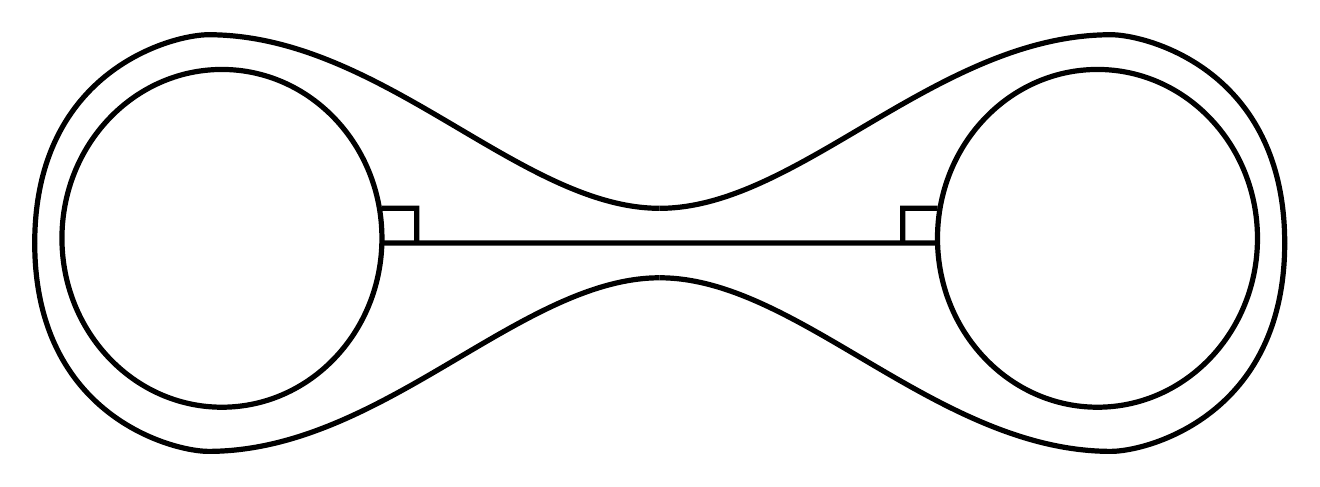,width=7cm,angle=0}\hspace{20pt}\epsfig{file =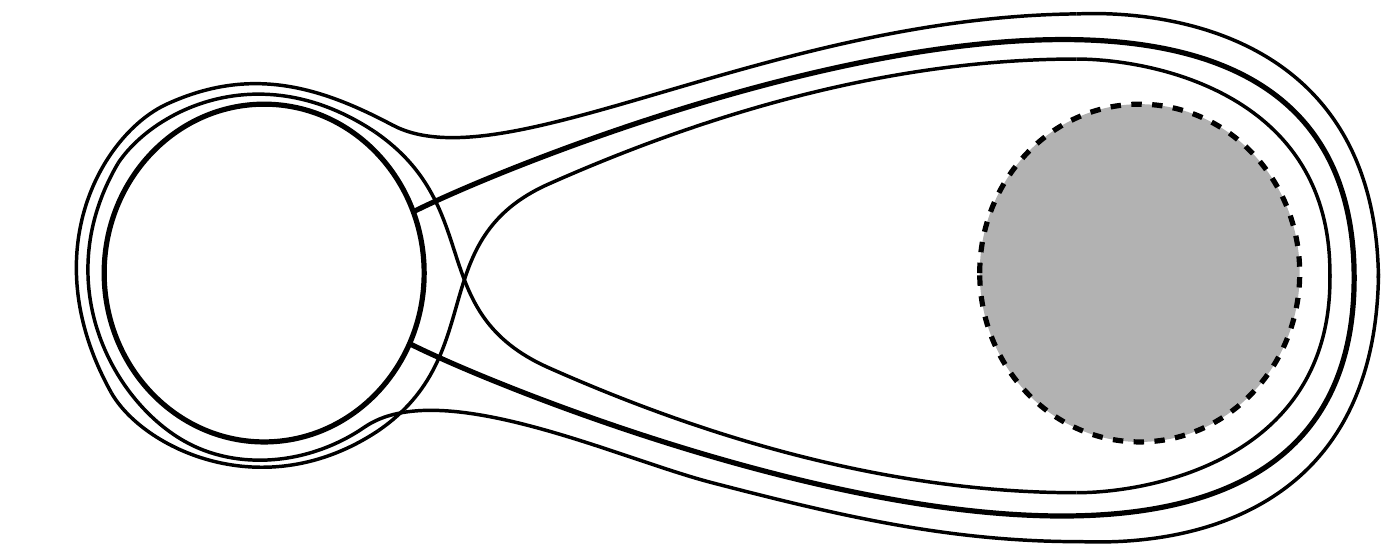,width=7cm,angle=0}}}
\vspace{-24pt}
\end{center}
}
\caption{Different types of chains and the associated embedded and immersed pants} \label{fig:pants}
\end{figure}

\begin{lemma}
Let $a \in \A$ be an arc between $\gamma_1, \gamma_2 \in \Gamma$. Then $\ell(a)$ is determined by $\ell(\gamma_1), \ell(\gamma_2)$ and $\ell(\gamma_a)$.
\end{lemma}
\begin{proof}
Note that if $\gamma_1$ and $\gamma_2$ are distinct, then they together with $\gamma_a$ are the three boundary curves of an embedded pair of pants. If they are not distinct then they are still the boundary curves of a pair of pants, but this time it's immersed and not embedded. The two cases are illustrated in Figure \ref{fig:pants}. 

In both cases, we can argue inside the pair of pants and use a standard fact from hyperbolic trigonometry that tells you that three lengths determine a right angled hyperbolic hexagon.
\end{proof}

Putting these two previous lemmas together, we have the following proposition. 

\begin{proposition}\label{prop:}
The quantities $\ell(\gamma),\tau(\gamma)$ for $\gamma\in \Gamma$ and $\ell(\gamma_a)$ for $a\in \A$ determine a surface $X\in \M_g$.
\end{proposition}

For future reference, we'll refer to the curve $\gamma_a$ for $a\in \A$ as the {\it chain} associated to $a$. The set of curves $\Gamma,\Gamma_\A:=\{\gamma_a\}_{a\in \A}$ will be referred to as a {\it curve and chain system}. 

We remark that the number of marked topological curve and chain systems, which we'll denote by $N_{cc}(g)$, is equal to the number of topological curve and arc decompositions with marked arcs. Thus, by Lemma \ref{lem:countca} we have the following. 

\begin{lemma}\label{lem:countcc}
The number $N_{cc}(g)$ of topological types of marked curve and chain systems is bounded above by 
$$
\frac{1}{e^6} \left( \frac{12^6}{e^5}\right)^{g-1} (g-1)^{6g-6}
$$
\end{lemma}

\subsection{Curve and chain systems of bounded length}\label{ss:}

In this section we prove the existence of curve and chain systems of lengths bounded above by a function of topology (Theorem \ref{thm:cclength}).

We'll need to bound the distance between a geodesic loop based in a point and the collar neighborhood of the corresponding simple closed geodesic (the core curve). This comes up in the following situation. If the core curve is not too short, say greater than $1$ for instance, then there is a bound on the (Hausdorff) distance between the loop and the closed geodesic than only depends on the length of the loop (see \cite[Lemma 2.3]{ParlierBers} for a precise statement). However, if the core curve is arbitrarily short, the loop can be arbitrarily far away. The following lemma gives a bound, that only depends on the length of the loop, on the distance between the collar of the core curve and the loop.

\begin{lemma}\label{lem:collardistance}
Let $c\subset X$ be a geodesic simple loop and $\gamma$ be the unique simple closed geodesic freely homotopic to $c$. Then 
$$
\sup_{p \in c} \{d_X(p,\CC(\gamma)\} <\log \left(\sinh\left( \frac{\ell(c)}{2} \right) \right)
$$
\end{lemma}

\begin{proof}
The proof is a straightforward hyperbolic trigonometry computation. The loop $c$ (based in a point $p$) and $\gamma$ form the boundary curves of an embedded cylinder in $X$. The cylinder can further be decomposed into two isometric quadrilaterals with three right angles (sometimes called trirectangles or Lambert quadrilaterals) with opposite sides of lengths $\sfrac{\ell(c)}{2}$ and $\sfrac{\ell(\gamma)}{2}$ as in Figure \ref{fig:quad}.

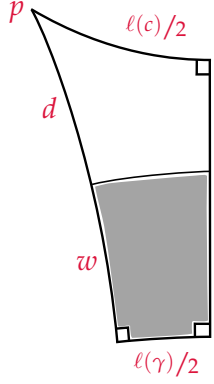
\begin{figure}[h]
{\color{linkred}
\leavevmode \SetLabels
\L(.503*.91) $\sfrac{\ell(c)}{2}$\\%
\L(.43*.67) $d$\\%
\L(.40*.97) $p$\\
\L(.46*.23) $w$\\
\L(.51*-0.07) $\sfrac{\ell(\gamma)}{2}$\\
\endSetLabels
\begin{center}
\AffixLabels{\centerline{\epsfig{file =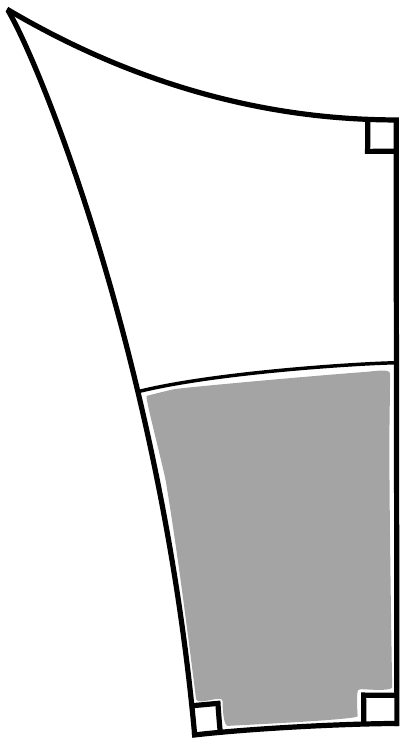,width=2.5cm,angle=0}}}
\vspace{-24pt}
\end{center}
}
\caption{The quadrilateral} \label{fig:quad}
\end{figure}

The collar around $\gamma$ corresponds to the shaded region in Figure \ref{fig:quad}. The segment marked $d$ is an upper bound on the distance between $c$ and $\CC(\gamma)$ and this is what we want to bound. It's a subarc of the arc of length $d + w$ joining $p$ to $\gamma$ where $w$ is the width of $\CC(\gamma)$. Appealing to hyperbolic trigonometry in the quadrilateral one obtains
$$
\sinh\left(\frac{\ell(c)}{2} \right) = \sinh\left(\frac{\ell(\gamma)}{2} \right) \cosh\left(d+\arcsinh\left(\frac{1}{\sinh\left(\sfrac{\ell(\gamma)}{2}\right)} \right)\right)
$$
From this we obtain
$$
d = \arccosh \left(\frac{\sinh\left(\frac{\ell(c)}{2} \right)}{\sinh\left(\frac{\ell(\gamma)}{2} \right)} \right) - \arcsinh\left(\frac{1}{\sinh\left(\sfrac{\ell(\gamma)}{2}\right)}\right)
$$
This becomes
\begin{eqnarray*}
d & = & \log \left( \frac{\sinh\left(\frac{\ell(c)}{2} \right)}{\sinh\left(\frac{\ell(\gamma)}{2}\right)}+ \sqrt{   \frac{\sinh^2\left(\frac{\ell(c)}{2} \right)}{\sinh^2\left(\frac{\ell(\gamma)}{2}\right)}  -1 }\right) - \log \left( \frac{1}{\sinh\left(\frac{\ell(\gamma)}{2}\right)}+ \sqrt{   \frac{1}{\sinh^2\left(\frac{\ell(\gamma)}{2}\right)}  + 1 }\right) \\
& < &  \log \left( \frac{2\sinh\left(\frac{\ell(c)}{2} \right)}{\sinh\left(\frac{\ell(\gamma)}{2}\right)}\right) - \log \left( \frac{2}{\sinh\left(\frac{\ell(\gamma)}{2}\right)}\right) \\
& = & \log \left(\sinh\left( \frac{\ell(c)}{2} \right) \right)
\end{eqnarray*}
as desired.
\end{proof}

We now prove the existence of short chain and curve systems, the main result of this section. 

\begin{proof}[Proof of Theorem \ref{thm:cclength}]
Consider $X\in \M_g$. We begin by considering the set of curves $\Gamma_0$ of $X$ of length at most $2 \arcsinh(1)$. By Corollary \ref{cor:shortboundary}, their collars are disjoint. We now consider 
$$
X_0 := X \setminus \{\CC(\gamma) \mid \gamma \in \Gamma_0 \}
$$
This set may not be connected. 

We'll now iterate the following step starting with $k=0$: we choose a point $x\in X_0$ such that
$$d_{X_k}(x,\partial X_k\} \leq \log(4g)$$
We consider the shortest non-trivial loop $\delta_x$ based in $x$. By Lemma \ref{lem:systole} we have a bound on its length:
$$
\ell(\delta_x) \leq 2 R_g
$$
It's straightforward to check that $R_g < \log(4g)$.

We consider the unique simple closed geodesic $\delta$ (on $X$) freely homotopic to $\delta_x$. By Lemma \ref{lem:collardistance}, the distance between $x$ and $\CC(\delta)$ satisfies
$$
d_{X_k} (x, \CC(\delta)) < \log(\sinh(R_g))< \log(2g)
$$
where the last inequality is the result of comparing the two functions by standard manipulations. In particular, this implies that $\delta$ is contained in $X_k$ and is not a boundary curve of $X_k$. We then set 
$$
X_{k+1}:=X_k
$$
and repeat the procedure until all $x\in X_k$ satisfy $d_{X_k}(x,\partial X_k) < \log(4g).$ The disjoint set of curves we've cut $X$ along (which include those of $\Gamma_0$) are denoted $\Gamma$. For all $\gamma \in \Gamma$ we have
$$
\ell(\gamma) < 2 \log(4g)
$$
as desired.

We now consider
$$X':= X_0\setminus \Gamma$$
and consider a Voronoi cell decomposition of $X$ around the curves of $\Gamma$. This is simply the attribution of (at least) one element of $\Gamma$ to every $x\in X'$ by taking the curve (or curves) of $\Gamma$ closest to $x$.
The cells of the decomposition are 
$$C_\gamma := \{x \in X' \mid d_{X'}(x, \gamma)\leq d_{X'}(x, \delta) \mbox{ for all }\delta \in \Gamma\}$$
Note that by construction
$$
d_{X'}(x, \gamma)\leq \log(4g)
$$
for all $x \in C_\gamma$.

The points that lie in several cells we refer to as the cut locus of the decomposition. It's an embedded graph, and generically the graph is trivalent. We want to find a decomposition of $X'$ into hexagons dual the cell decomposition as follows. This process is completely analogous to the construction of a Delaunay triangulation for a choice of points in the plane for instance.

Dual to each edge of the cut locus we construct an arc between the corresponding curves of $\Gamma$ as in Figure \ref{fig:arc}. 

\begin{figure}[h]
{\color{linkred}
\leavevmode \SetLabels
\endSetLabels
\begin{center}
\AffixLabels{\centerline{\epsfig{file =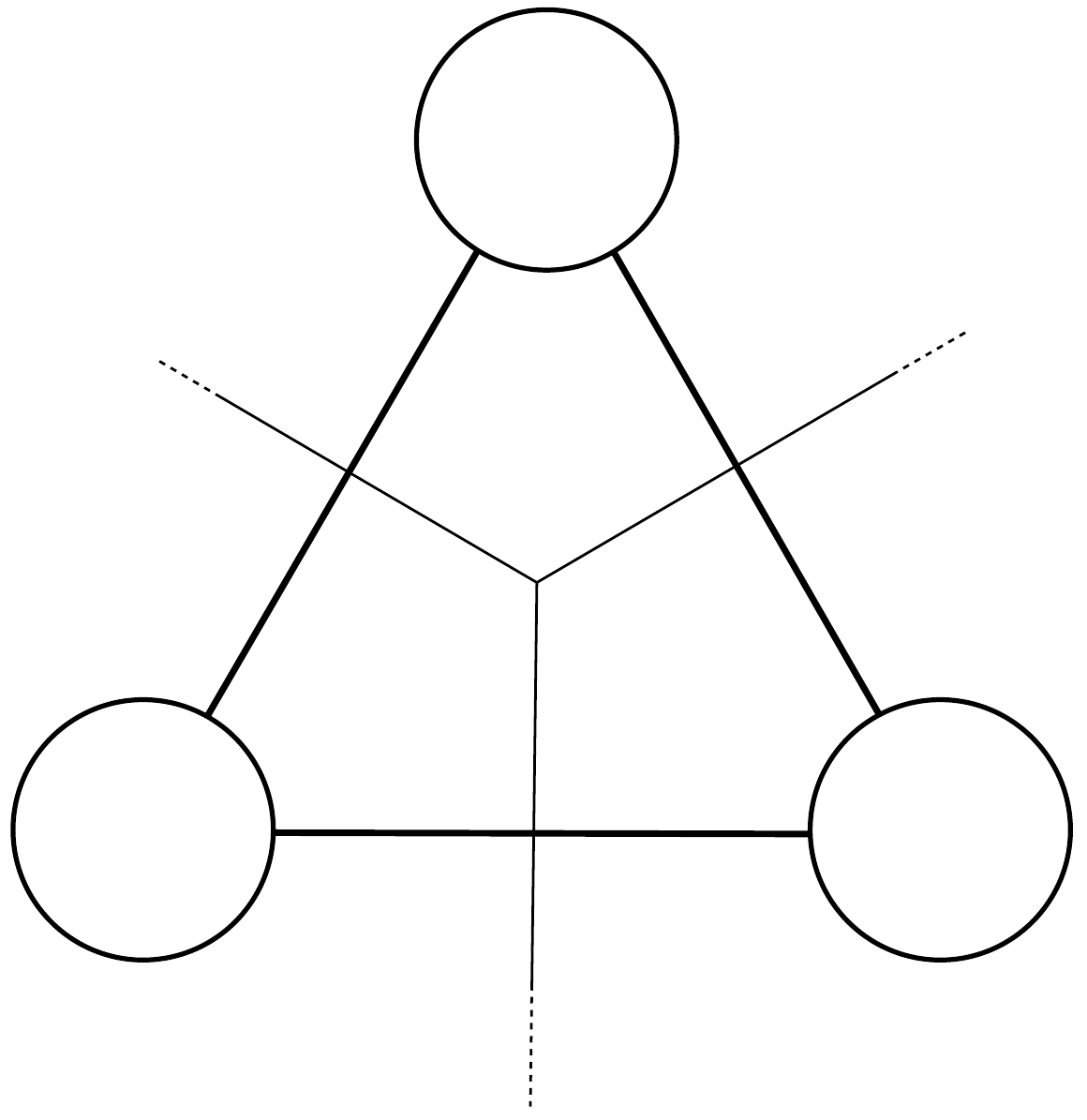,width=3.5cm,angle=0}}}
\vspace{-24pt}
\end{center}
}
\caption{Constructing arcs dual to the cut locus} \label{fig:arc}
\end{figure}

If the cut locus is trivalent (which it is generically) then the resulting arc decomposition is maximal in that all remaining (non trivial and non homotopic) arcs essentially cross these. As such, this provides a decomposition into hexagons. 

If the cut locus is not trivalent, there are choices to be made (just like when a set of points in the plane admits several Delaunay triangulations). To do so, in any vertex $v$ of degree $k\geq 4$ of the cut locus, consider the set of simple arcs $\{ c_i\}_{i=1}^k$ to each of the boundary curves of $X'$ whose Voronoi cells touch $v$. Each arc is contained in the corresponding cell and we suppose that they are cyclically oriented around $v$. Now fix one of these arcs, say $c_1$, and orient it towards $v$. Orient all of the others away from $v$ and consider the arcs obtained by concatenating $c_1$ with $c_i$ for $i=3,\hdots,k-1$ (see Figure \ref{fig:cutlocus}).

\begin{figure}[h]
{\color{linkred}
\leavevmode \SetLabels
\L(.170*.37) $v$\\%
\L(.147*.27) $c_1$\\
\L(.22*.285) $c_2$\\%
\L(.23*.465) $c_3$\\
\L(.183*.62) $c_4$\\%
\L(.110*.476) $c_5$\\
\endSetLabels
\begin{center}
\AffixLabels{\centerline{\epsfig{file =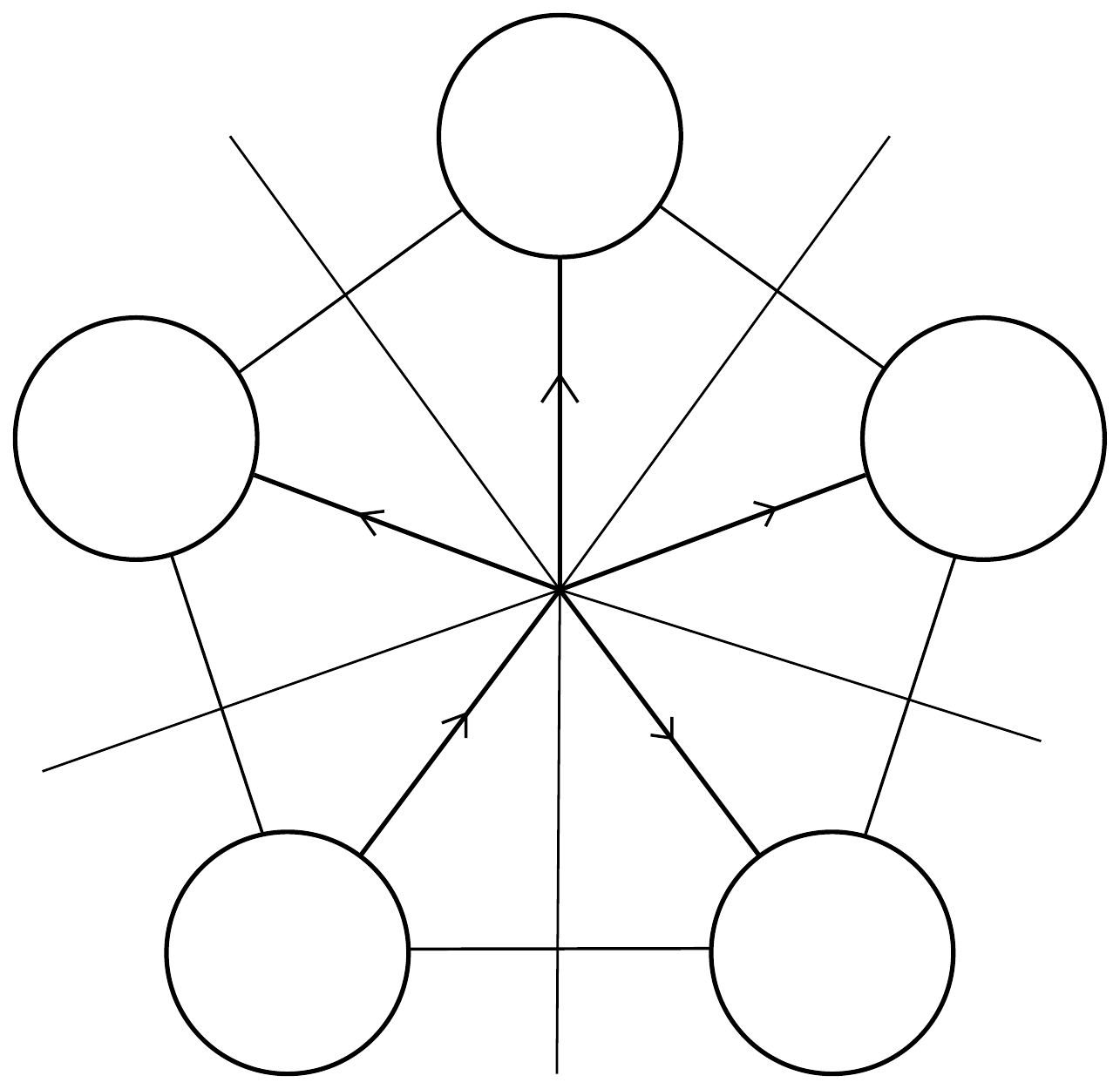,width=4.5cm,angle=0}\hspace{10pt}\epsfig{file =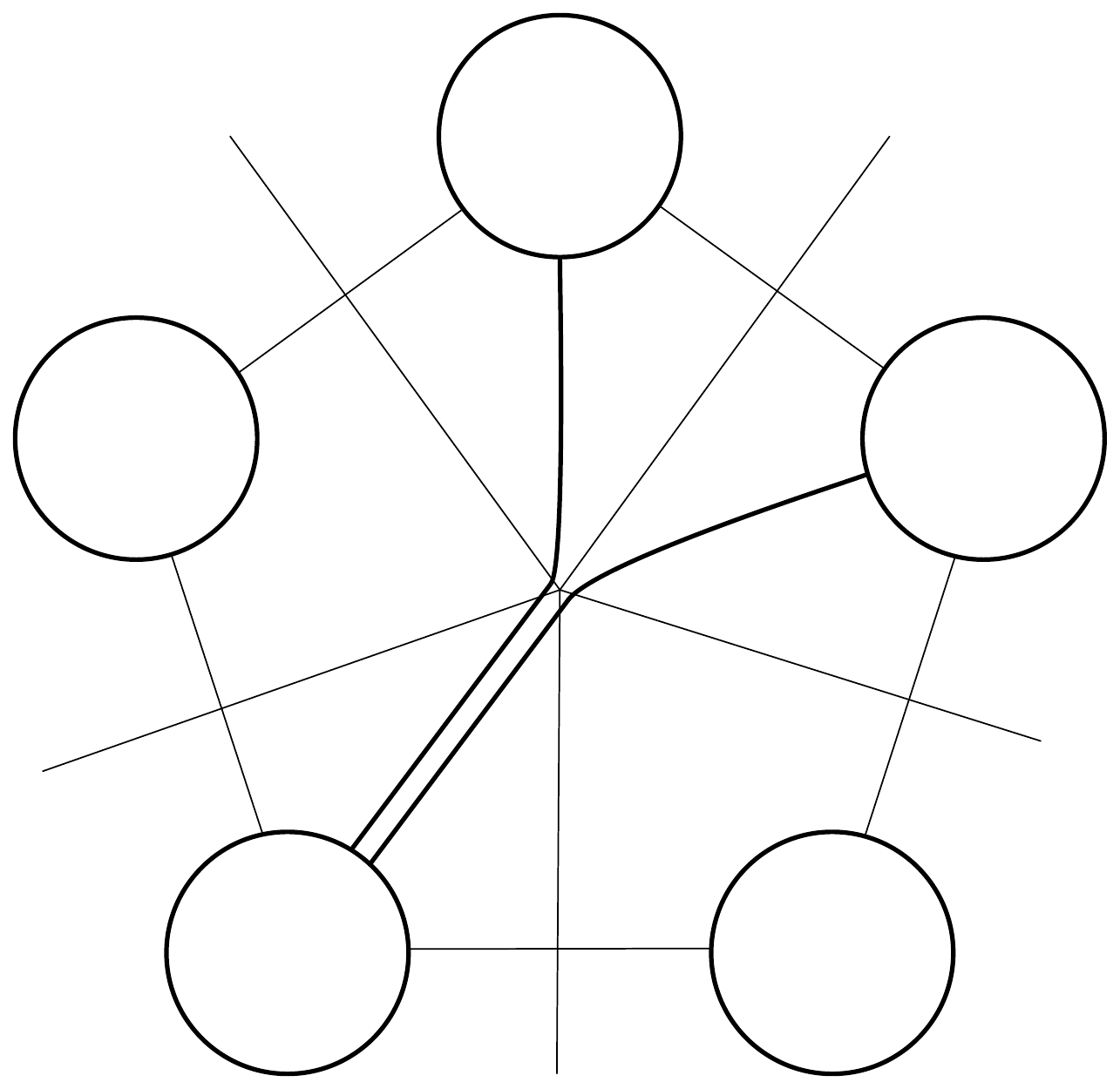,width=4.5cm,angle=0}\hspace{10pt}\epsfig{file =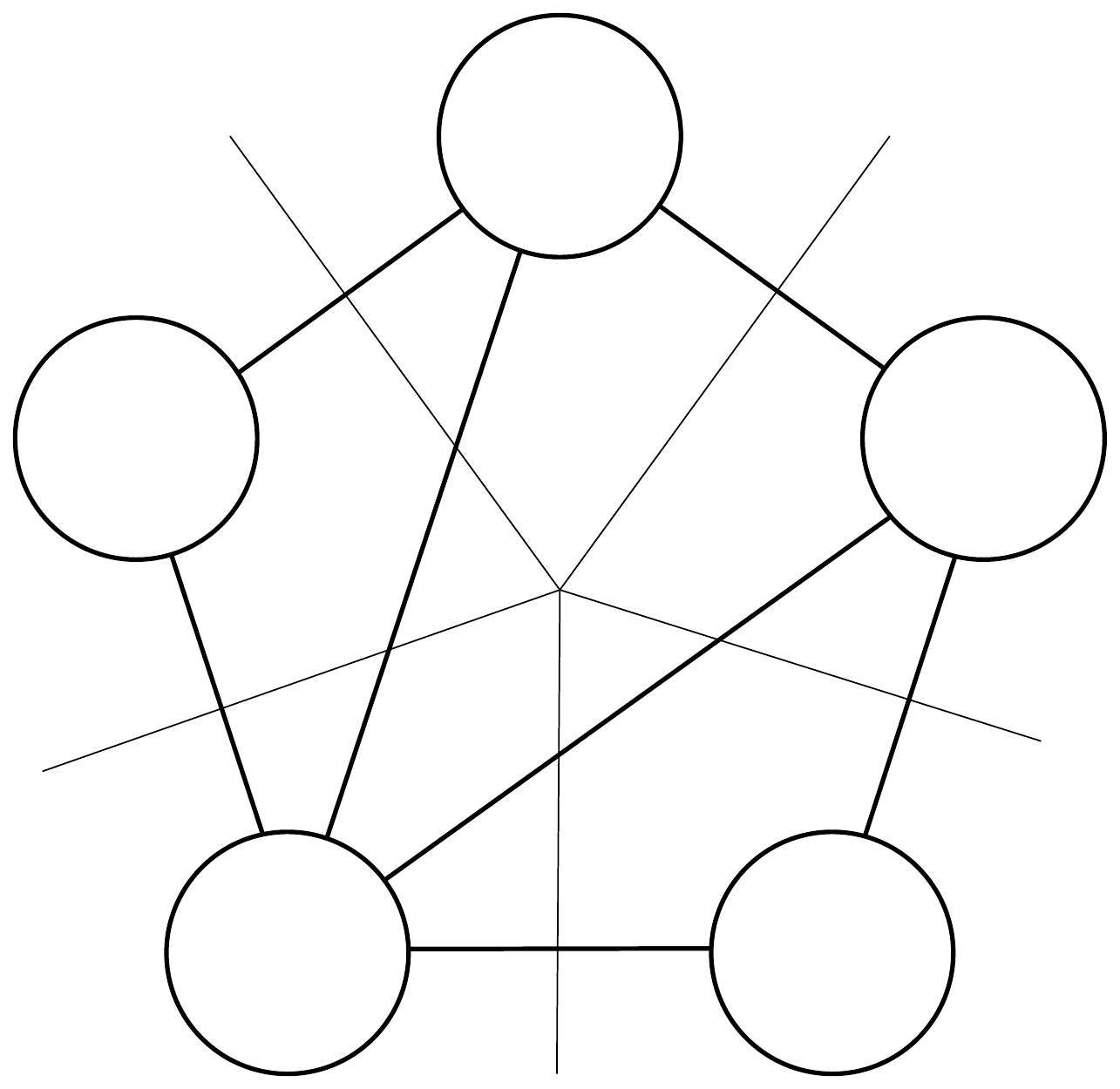,width=4.5cm,angle=0}
}}
\vspace{-24pt}
\end{center}
}
\caption{Constructing arcs} \label{fig:cutlocus}
\end{figure}
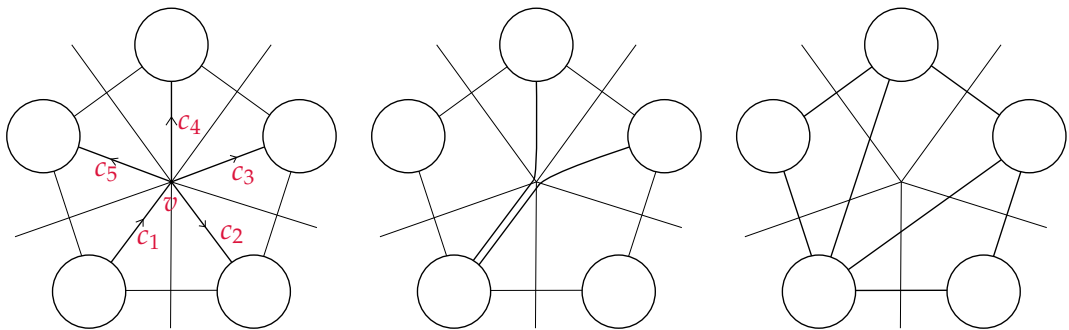

Note that the arcs $c_1* c_2$ and $c_1 * c_k$ are isotopic to arcs dual to the cut locus. We add this set of $k-3$ arcs and repeat the process in all vertices of degree higher than $3$.

Once we've chosen these arcs, we take the minimizers that join the corresponding boundary curves in $X'$ and these are simple orthogeodesics the set of which we denote by $\A$. By construction, any $a \in \A$ satisfies
$$
\ell(a) \leq 2 \log(4g)
$$
Now consider the curves $\gamma_a$ for $a\in \A$ defined in Section \ref{sec:curveandchain}. All that remains to show is that they have length bounded by $8 \log(4g)$. 
In particular, for $a\in \A$, let the end geodesics of $a$ be $\gamma_1$ and $\gamma_2$, and by construction they are both of length at most $2 R_g$. Recall $\gamma_a$ is in the free homotopy class of $\gamma_1 * a * \gamma_2 * a^{-1}$ and because $\ell(a) \leq \log(2g)$, we have
$$
\ell(\gamma_a) < 4 R_g + 4 \log(4g) < 8 \log(4g)
$$
proving the result.
\end{proof}

Before passing to the next section, we prove a lemma that we will need to control twist parameters. Here the sets $\Gamma$ and $\Gamma_\A$ are the short curve and chain system from the theorem above. The goal is to do the following: for each $\gamma \in \Gamma_1:=\Gamma \setminus \Gamma_0$ we want to choose a transversal curve $\delta_\gamma$ that is not too long. The specific result we prove is the following.

\begin{lemma}\label{lem:trans}
For each $\gamma \in \Gamma \setminus \Gamma_0$ there exists a simple closed geodesic $\delta_\gamma$ such that $i(\gamma,\delta_\gamma) \leq 2$
and 
$$
\ell(\delta_\gamma) <14 \log(4g)
$$
Furthermore there is a choice of such a $\delta_\gamma$ that is only determined by the topological type of a marked curve and chain system.
\end{lemma}
\begin{proof}
Consider $\gamma \in \Gamma \setminus \Gamma_0$ and consider its two copies on $X_0\setminus \Gamma$, say $\gamma_1$ and $\gamma_2$. Note $\gamma$ can be given an orientation given by an orientation of the surface. 

If there is an arc $a \in \A$ joining $\gamma_1$ to $\gamma_2$, then consider the closed curve $\delta_\gamma$ on $X$ obtained by concatenating $a$ with the oriented sub-arc of $\gamma$ between the two endpoints of $a$ on $\gamma$ (see Figure \ref{fig:oneholedtorus}).

\begin{figure}[h]
{\color{linkred}
\leavevmode \SetLabels
\L(.31*.25) $a$\\%
\L(.17*.375) $\gamma$\\
\L(.734*.27) $\delta_\gamma$\\%
\endSetLabels
\begin{center}
\AffixLabels{\centerline{\epsfig{file =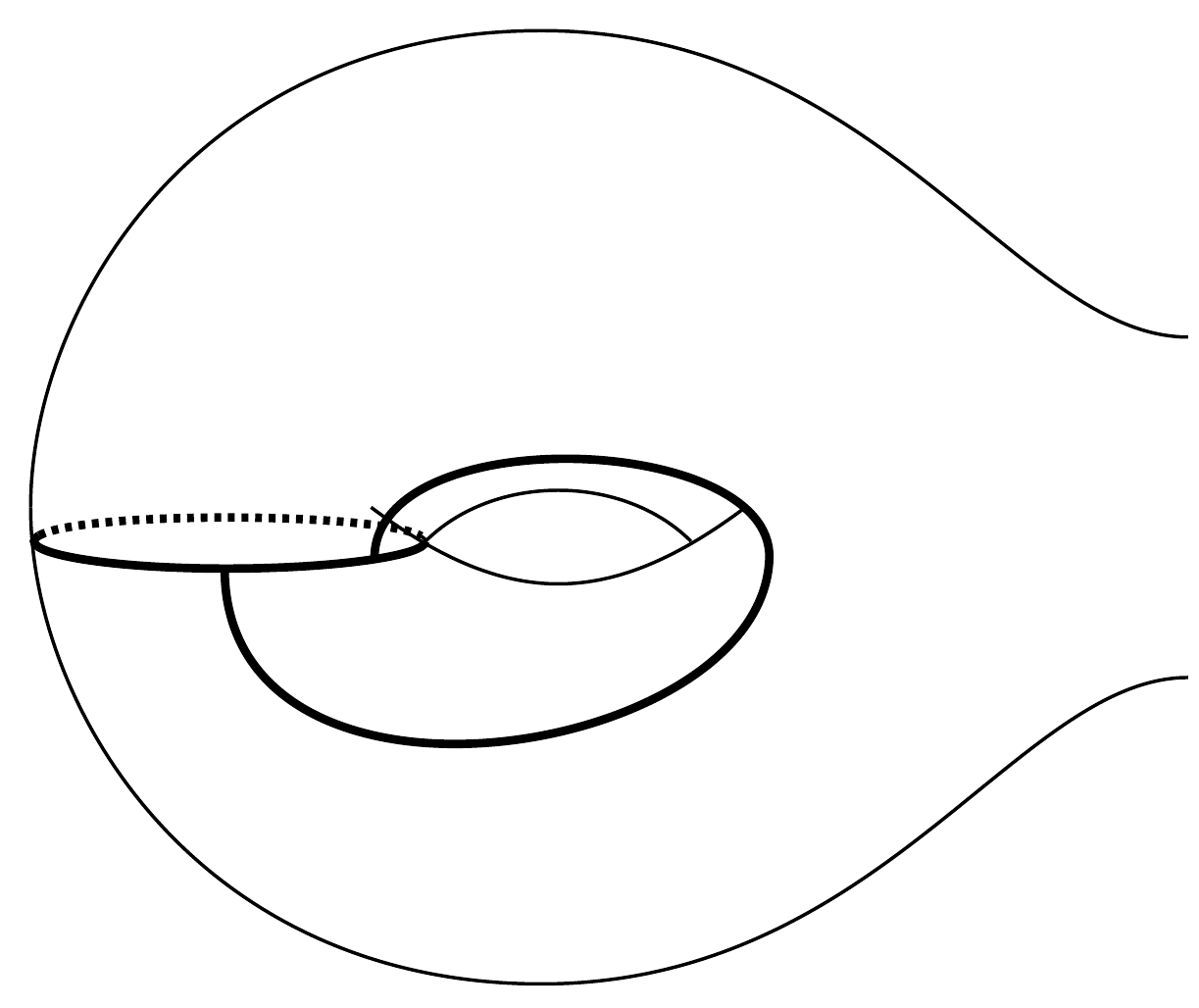,width=4.5cm,angle=0}\hspace{50pt}\epsfig{file =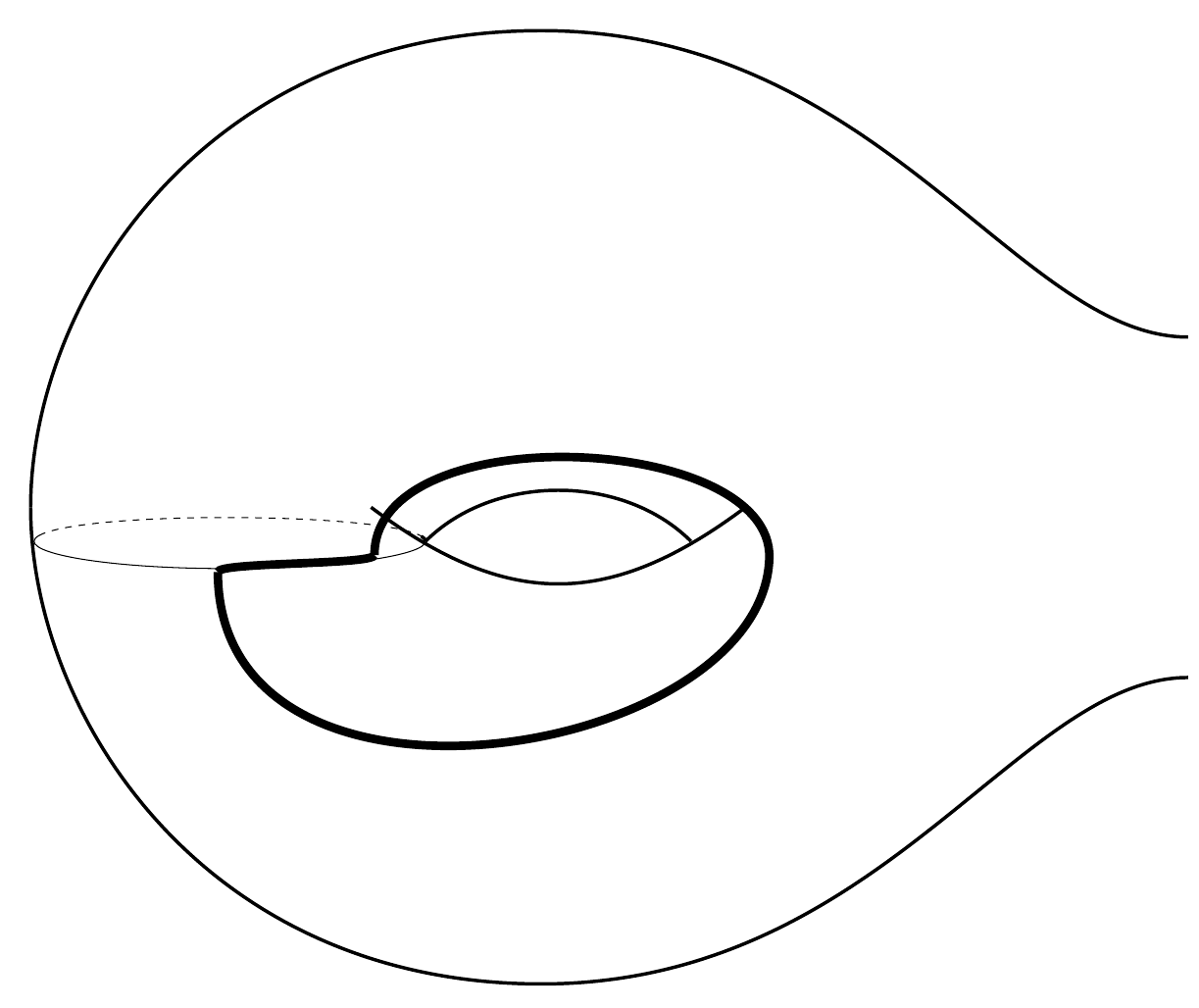,width=4.5cm,angle=0}}}
\vspace{-24pt}
\end{center}
}
\caption{The one holed torus and the construction of $\delta_\gamma$} \label{fig:oneholedtorus}
\end{figure}

By construction $\ell(\delta_\gamma) \leq \ell(a) +  \ell(\gamma) < 2\log(4g) + 2\log(4g) = 4 \log(4g)$
and $i(\gamma,\delta_\gamma)=1$.

We could possibly do better, by taking the shortest one, but we want to make a choice that only depends on the topology of a marked curve and chain system.

If no such arc exists, then for $\gamma_i$, for $i=1,2$ we construct an arc that has both endpoints on $\gamma_i$. Either there is an arc $a_i \in \A$ which does this, in which case, we use it, or no such arc exists in which case we consider any arc $a_i\in \A$ with an endpoint on $\gamma_i$. The other endpoint of $a_i$ must lie on another $\gamma'_i \in \Gamma$. Giving the arc and $\gamma'_i$ the appropriate orientation, we construct a homotopy class of arc $b_i$ by considering the concatenation
$$
a_i * \gamma'_i *a_i^{-1}
$$
Note that $b_1,b_2$ and $\gamma$ fill a four holed sphere. We also have
$$
\ell(b_i) < 6 \log(4g)
$$
using the bound on lengths of the concatenated paths. Now the shortest curve that essentially intersects $\gamma$ exactly twice on this four holed sphere has length at most
$$
\ell(b_1) +\ell(b_2) + \ell(\gamma)
$$
by the same type of cut-and-concatenate argument as before (see Figure \ref{fig:fourholedsphere}). 
\begin{figure}[h]
{\color{linkred}
\leavevmode \SetLabels
\L(.23*.61) $b_1$\\%
\L(.33*.24) $b_2$\\%
\L(.20*.39) $\gamma$\\
\L(.74*.27) $\delta_\gamma$\\%
\endSetLabels
\begin{center}
\AffixLabels{\centerline{\epsfig{file =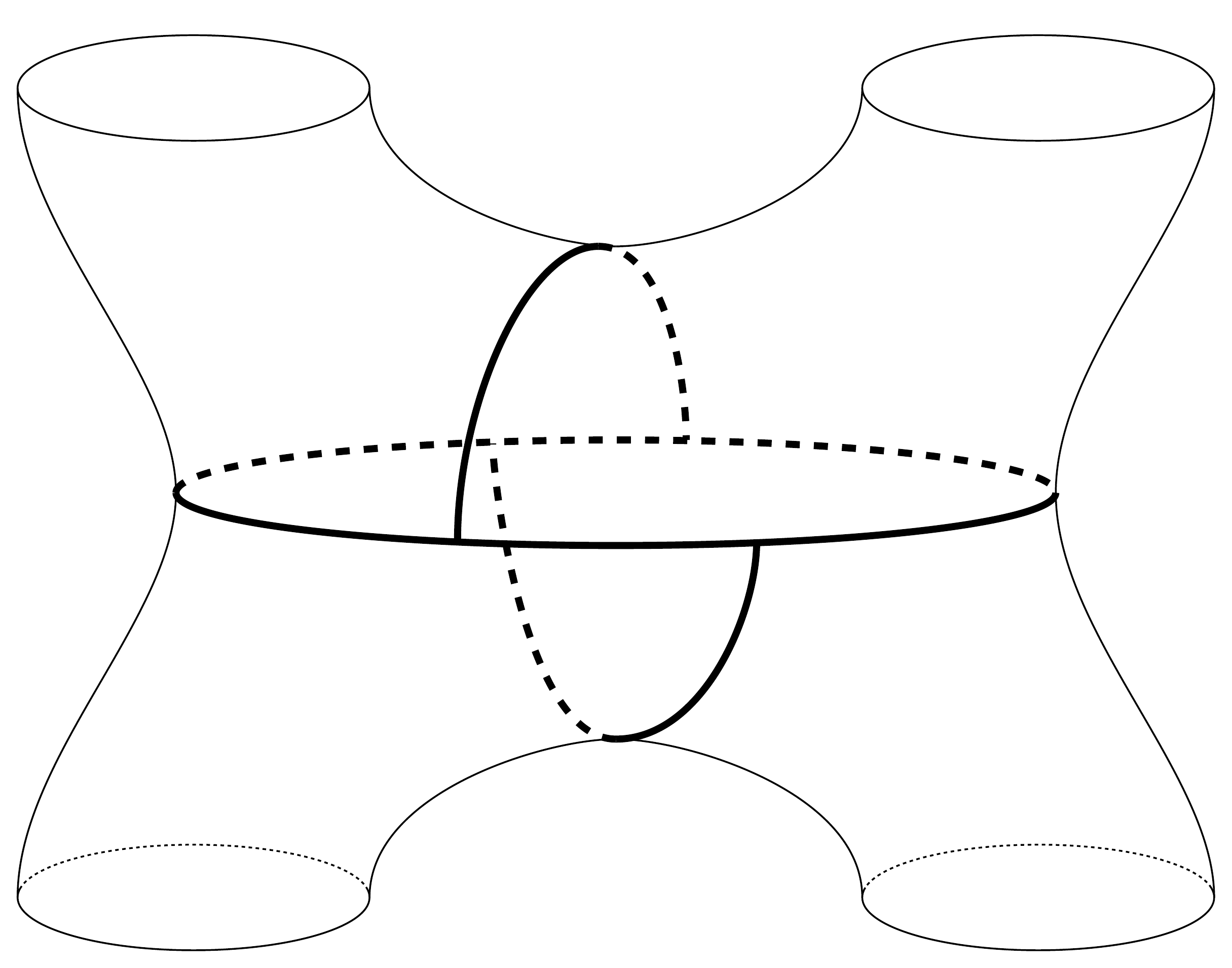,width=5.5cm,angle=0}\hspace{20pt}\epsfig{file =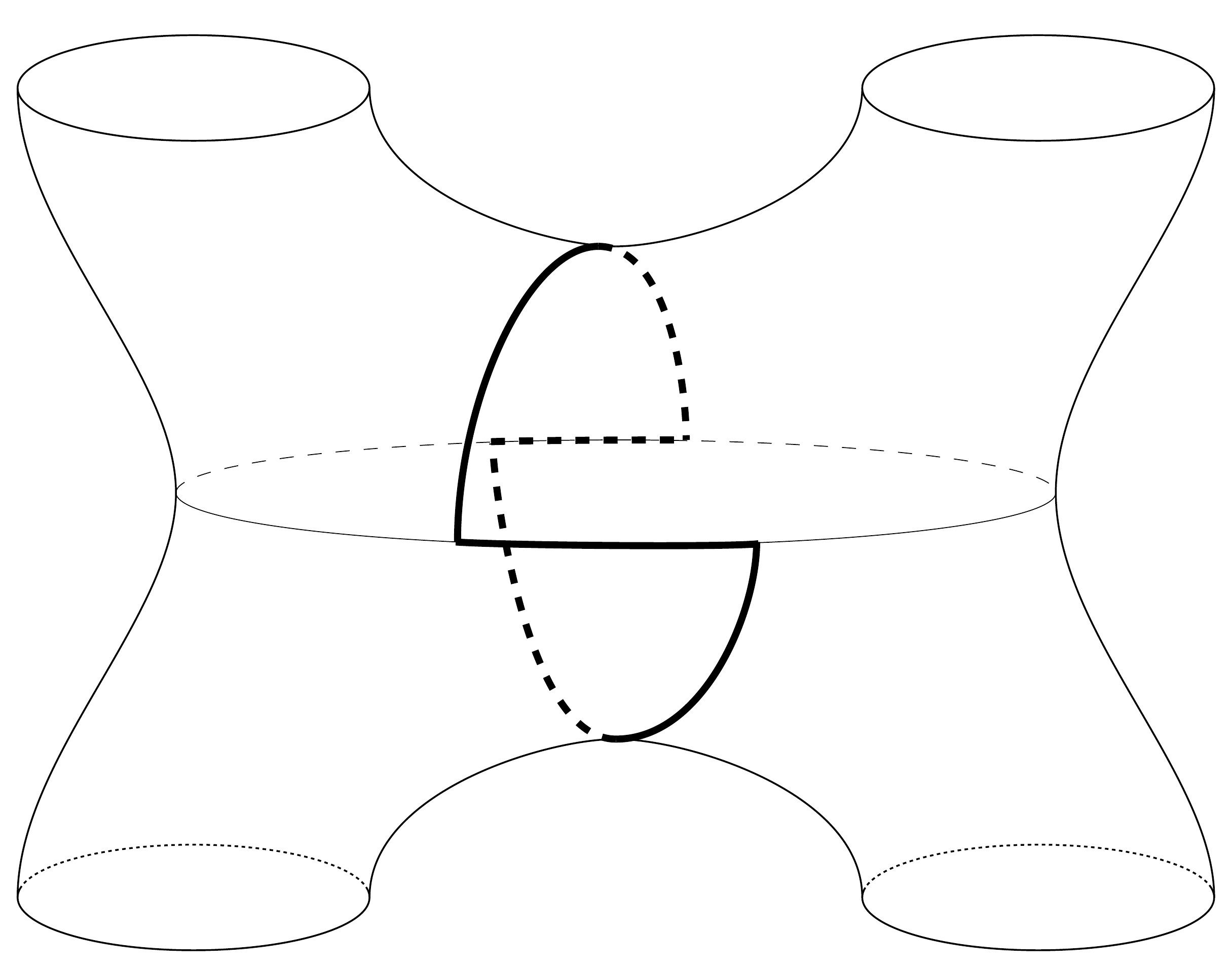,width=5.5cm,angle=0}}}
\vspace{-24pt}
\end{center}
}
\caption{The four holed sphere and the construction of $\delta_\gamma$} \label{fig:fourholedsphere}
\end{figure}

The resulting curve we denote $\delta_\gamma$ and we have
$$
\ell(\delta_\gamma) < 14 \log(4g)
$$
as desired.
\end{proof}

We'll denote by $\Delta_{\Gamma_1}$ the set of curves obtained as in the above proof. Note they are determined if we know the topological type of the curve chain decomposition $\Gamma, \Gamma_\A$. 

\section{Dealing with thin structures using the topology of arcs and length identities}\label{sec:thin}

Here we show how to identify, up to quantifiable finiteness, the isometry class of a surface using the isometry type of its thick part. The reason this requires a different analysis is because we have to determine a twist parameter using curves that could be arbitrarily long. This means we can't use estimates based on their length.  

Recall that given $X\in \M_g$, we've denoted by $\Gamma_0(X)$ the set of closed geodesics of length less than $2 \arcsinh(1)$. We denote by
$$
X_0 := X \setminus \Gamma_0
$$
This is a slight abuse of notation as we used $X_0$ before for surface obtained by removing the collars of the short curves, but as they determine each other, for simplicity we'll use it here too. The goal is to determine $X$ using $X_0$ and $\Lambda(X)$, and to do so we'll proceed one curve at a time.

\subsection{The topology of the next shortest curve}

Let $Y\subset X$ be a subsurface obtained by cutting along some subset of $\Gamma_0$ and let $\alpha \subset Y$ that belongs to $\Gamma_0$. Set $Y_\alpha := Y \setminus \alpha$. Denote by $\alpha_1,\alpha_2$ the two copies of $\alpha$ on $Y_\alpha$. 
We want to determine $Y$ knowing $\Lambda(Y)$ and $Y_\alpha$. 

Given $Y_\alpha$ and $\Lambda(Y)$, we know the set $\Lambda(Y) \setminus \Lambda(Y_\alpha)$. The first element of this set corresponds to the shortest closed geodesic of $Y$ which is not entirely contained in $Y_\alpha$. Let us denote the corresponding closed geodesic $\alpha'$ and observe that necessary $i(\alpha,\beta) \neq 0$. First we discuss the topology of $\alpha'$. 

\begin{lemma}\label{lem:nexttop}
The curve $\alpha'$ is a simple closed geodesic which is one of two types:

- Type (I): $i(\alpha,\alpha') = 1$ and the two curves fill a one holed torus,

- Type (II): $i(\alpha,\alpha') = 2$ and the two curves fill an embedded four holed sphere.
\end{lemma}

\begin{proof}
This will all follow from standard cut and paste arguments and the fact that $\ell(\alpha)\leq 2 \arcsinh(1)$.

Consider the image of $\alpha'$ on $Y_\alpha$ is a collection of arcs $a_1,\hdots,a_m$ with endpoints on $\alpha_1 \cup \alpha_2$. In the course of the proof, we'll show that the arcs are all simple. Note that if there is only one and it is simple, then it lies between $\alpha_1$ and $\alpha_2$, thus $i(\alpha,\alpha') = 1$ and $\alpha'$ is of type (I). 

Note that because $\ell(\alpha)\leq 2 \arcsinh(1)$, by Lemma \ref{lem:collar} there is an embedded collar of width $\arcsinh(1)$ around $\alpha$. In particular, any of these arcs is of length at least $2\arcsinh(1)$. 

In particular that if there is an arc $a_i$ between $\alpha_1$ and $\alpha_2$, then it is the only arc. Indeed, if there was another one, then $\ell(\alpha') > 2\arcsinh(1) + \ell(a_i)$. The distance between $\alpha_1$ and $\alpha_2$ is at most $\ell(a_i)$. Denote by $a$ a shortest distance path between $\alpha_1$ and $\alpha_2$. Note that $a$ is simple and $\ell(a) \leq \ell(a_i)$. 

Now by concatenating $a$ and a subarc of $\alpha$, one can construct an essential curve $\alpha''$ that intersects $\alpha$ exactly once. And as $\ell(\alpha) \leq 2 \arcsinh(1)$, this implies that
$$
\ell(\alpha'') \leq 2 \arcsinh(1) + \ell(a)\leq 2 \arcsinh(1) + \ell(a_i)<\ell(\alpha')
$$
a contradiction. This proves that if $a_1,\hdots,a_m$ contains an arc between $\alpha_1$ and $\alpha_2$, then $m=1$, it is simple and $\alpha'$ is of type (I).

If not, then all arcs of $a_1,\hdots,a_m$ have both endpoints on either $\alpha_1$ or $\alpha_2$. Denote by $a_1,\hdots,a_{m_1}$ those with endpoints on $\alpha_1$. Note that $m=2m_1$ as the number of endpoints of the arcs must be the same on both $\alpha_1$ and $\alpha_2$.

{\it Claim:} Any arc $a\in \{a_1,\hdots,a_m\}$ must be simple. 

{\it Proof of claim:} The basic idea is to do surgery in a point of self-intersection, but we want to be careful to ensure that the surgery doesn't produce a trivial arc. First observe that by minimality of $\alpha'$, $a$ is a shortest non-trivial arc between its endpoints (both on $\alpha_i$ for $i\in \{1,2\}$). Now if $a$ is not simple, it contains as a subarc a simple embedded loop $\hat{a}$ based in a point $p$ (this is true of any non-trivial non-simple arc). So the two subarcs of $a$ between $\alpha_i$ and $p$ are distance realizing and hence simple and disjoint. Denote them by $d_1$ and $d_2$. We now give orientations to $a, d_1,d_2$ and $\hat{a}$, so that $a = d_1 * \hat{a} *d_2$. The arc $a' = d_1 * \hat{a}^{-1}*d_2$ is non-trivial and the unique geodesic in its homotopy class (with endpoints fixed) is shorter than $a$, a contradiction and this proves the claim.

For $i=1,2$, we consider the shortest non-trivial arc $b_i$ with both of its endpoints on $\alpha_i$. As above $b_1$ and $b_2$ are simple and by minimality they satisfy the following: $\ell(b_1) \leq \ell(a_j)$ for $j\leq m_1$ and $\ell(b_2) \leq \ell(a_j)$ for $j> m_1$.


We now observe that $i(b_1,b_2)=0$. If not, by concatenating the appropriate subarcs, we can construct a path of length at most $\max_{i\in\{1,\hdots,m\}}\{\ell(a_i)\}$ between $\alpha_1$ and $\alpha_2$. The same argument used when there was an arc $a_i$ between $\alpha_1$ and $\alpha_2$ can be repeated, and we reach a contradiction. Similarly, if $m_1>1$, then we can also reproduce this argument.

To resume, this shows that unless $\alpha'$ is of type (I), the image of $\alpha'$ on $Y_\alpha$ consists of two simple disjoint arcs $a_1, a_2$ with base points on $\alpha_1$ and $\alpha_2$. Thus $\alpha,\alpha'$ fill a four holed sphere on $Y$ as claimed and $\alpha'$ is of type (II).
\end{proof}

We now deal with counting possible homotopy classes for $\alpha'$ of Type (I).

\subsubsection{Type (I) next shortest curves}

If $\alpha'$ is of type (I), then the projection (or restriction) of $\alpha'$ to $Y_\alpha$ is a simple arc between $\alpha_1$ and $\alpha_2$. Denote by $a$ the unique orthogeodesic in its free homotopy class. Given $a$, we can basically reconstruct $\alpha'$ using the fact that $\alpha'$ is the shortest closed geodesic intersecting $\alpha$. Indeed, $\alpha'$ is homotopic to the concatenation of the (or a) shortest subarc of $\alpha$ between the two endpoints of $a$. This is because the length of $\alpha'$ is a function of the length of the subarc of $\alpha$ and the length of $a$, and is monotonic increasing in both quantities. As there are at most two shortest subarcs of $\alpha$ between the two endpoints of $a$, knowing $a$ will determine one of two possible homotopy classes for $\alpha'$. 

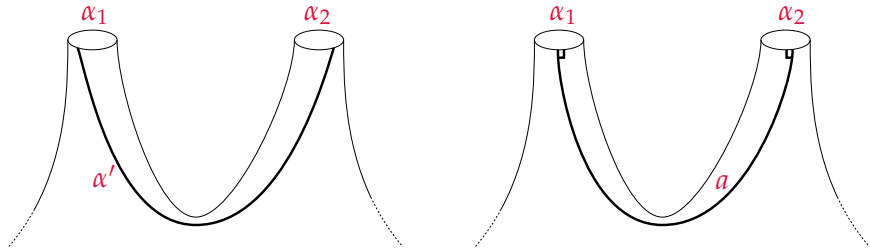
\begin{figure}[h]
{\color{linkred}
\leavevmode \SetLabels
\L(.185*1.02) $\alpha_1$\\%
\L(.38*1.02) $\alpha_2$\\%
\L(.595*1.02) $\alpha_1$\\%
\L(.795*1.02) $\alpha_2$\\%
\L(.195*.28) $\alpha'$\\
\L(.74*.27) $a$\\%
\endSetLabels
\begin{center}
\AffixLabels{\centerline{\epsfig{file =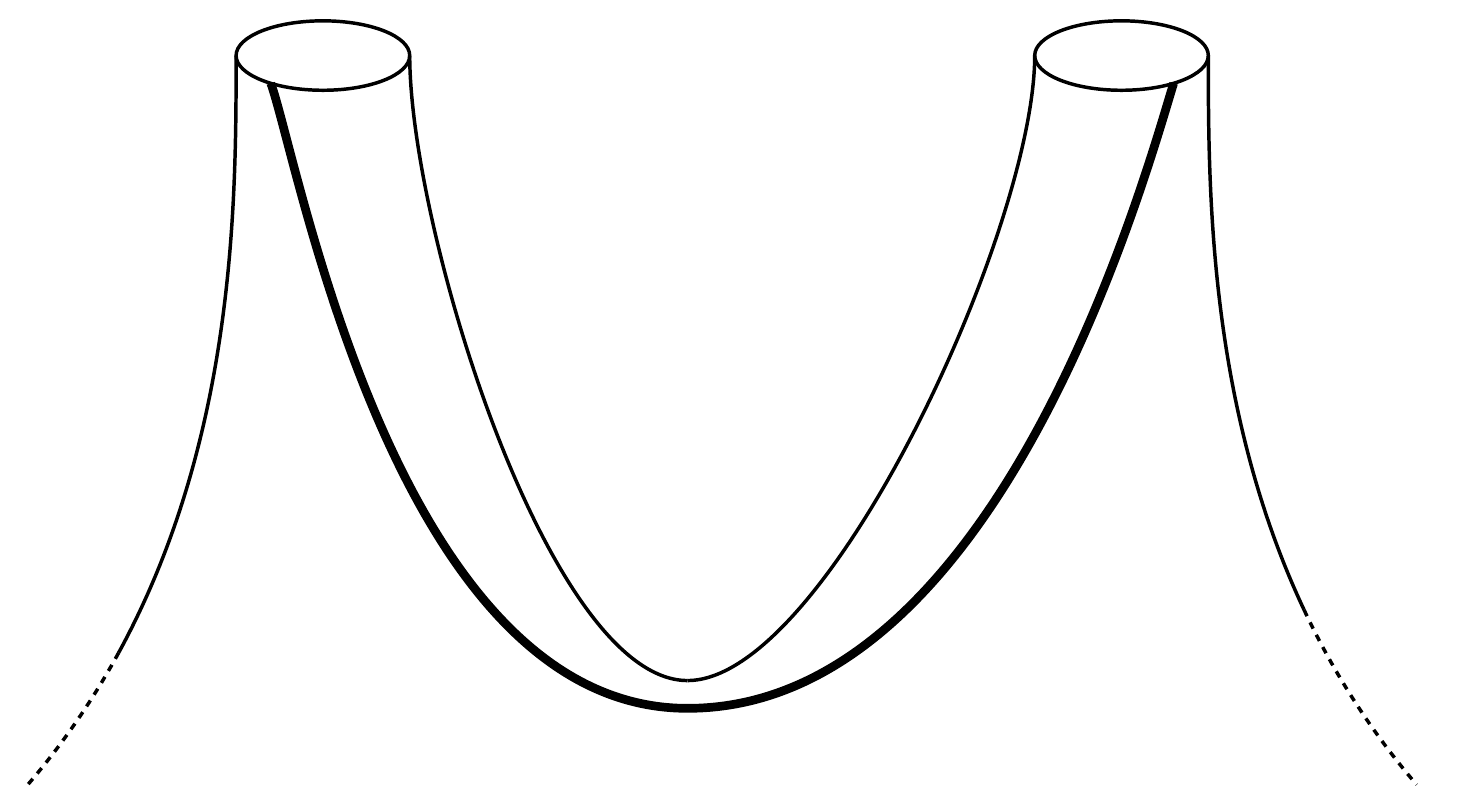,width=5.5cm,angle=0}\hspace{20pt}\epsfig{file =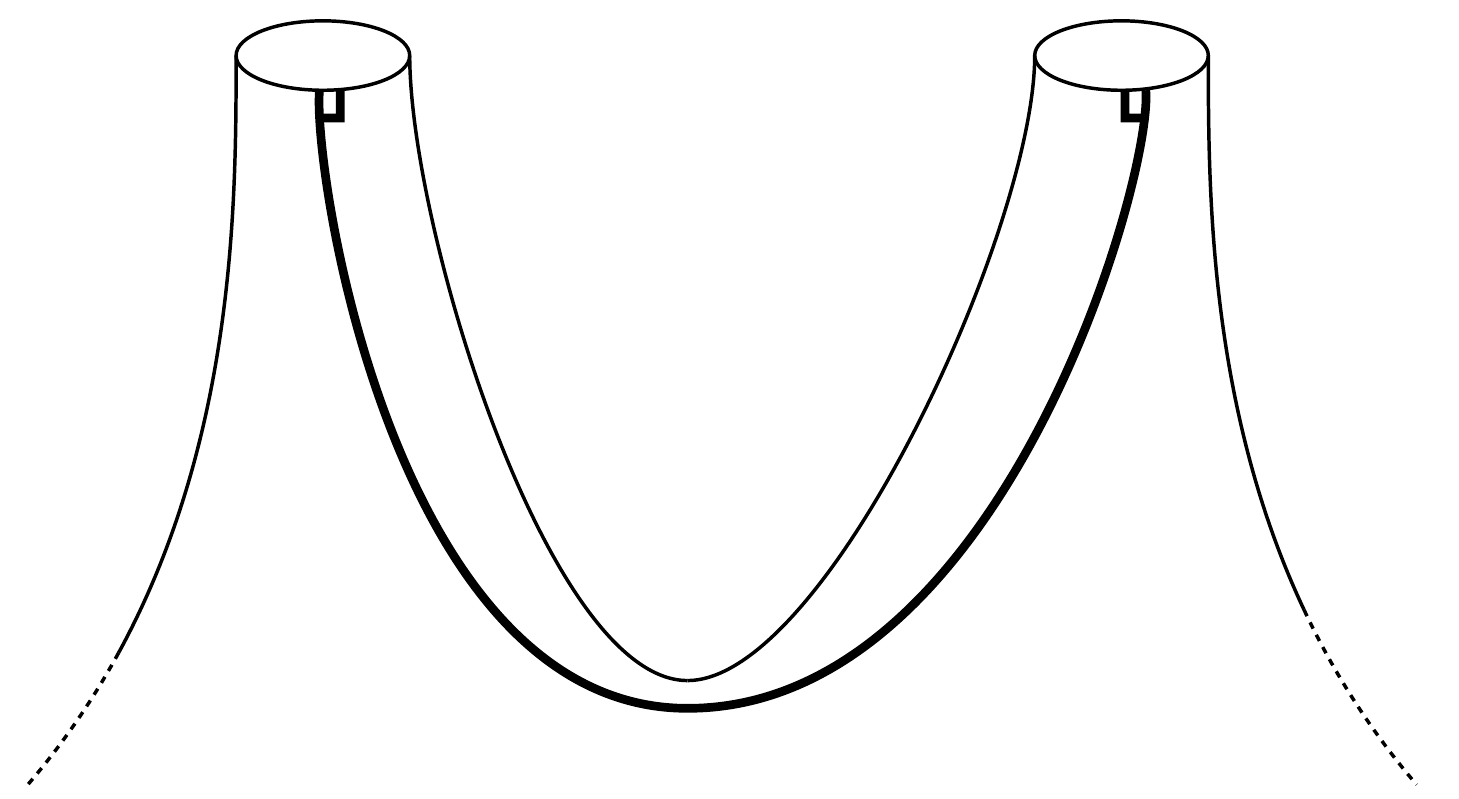,width=5.5cm,angle=0}}}
\vspace{-24pt}
\end{center}
}
\caption{When $\alpha'$ is of type (I)} \label{fig:typeIarcs}
\end{figure}

To resume the above discussion, if we determine all possible homotopy classes of orthogeodesics $a$, this will allow us to determine all possible homotopy classes of $\alpha'$. 

This observation will allow us to use a result of Przytycki \cite{Przytycki}. 

\begin{proposition}\label{prop:typeI}
There are at most $
(8g-8)(2g-1)<16(g-1)^2$ type (I) next shortest curves hence at most $
2(8g-8)(2g-1)< 32 (g-1)^2$ isometry classes of $Y$ when $\alpha'$ is of type (I). 
\end{proposition}
\begin{proof}
To show this, we consider all possible surfaces $Y^t_\alpha$ obtained from $Y_\alpha$ by fixing the twist parameter of $\alpha$ to be $t$ ($t\in [0.\ell(\alpha)]$). Among all of these, we consider the subset of them (possibly all or none of them) where $\alpha'$ is of type (I). For given $t$, we denote by $\alpha'_t$ the curve $\alpha'$. For each of these we look at the restriction of $\alpha'_t$ to $Y_\alpha \subset Y_\alpha^t$. It is a geodesic arc $a_t$ with endpoints on $\alpha_1$ and $\alpha_2$. 

This arc $a_t$ has the following property: on $Y_\alpha$ it is the shortest paths between its endpoints for the following reason. If there was a shorter path between its endpoints, this would give rise to a shorter curve on $Y^t_\alpha$ which intersects $\alpha$ exactly once. And, by the bigon property, such a curve is both essential and essentially intersects $\alpha$, contradicting the minimality of $\alpha'_t$. 

Consider two different such arcs, $a_{t}$ and $a_{s}$ coming from curves $\alpha'_t$ and $\alpha'_s$. Because both are minimal arcs between their endpoints, they cannot cross more than once. Thus $i(a_t, a_s) \leq 1$ for all $t,s \in [0,\ell(\alpha)]$. 

Now we pass to free homotopy classes, relative to boundary, of arcs, or equivalently to orthogeodesics in the homotopy classes of our collection of arcs. It's a result of Przytycki \cite{Przytycki} that their are at most $2 |\chi(Y_\alpha) | (| \chi(Y_\alpha) | +1)$ such homotopy classes. By the discussion above, this gives rise to at most 
$$
4|\chi(Y_\alpha) | (| \chi(Y_\alpha) | +1)
$$
possible homotopy classes for $\alpha'$. It's important to note here that we can do more than just bound their number. Our knowledge of $Y_\alpha$ allows us to determine the homotopy classes exactly.

Now to determine $Y$, for each possible homotopy class of $\alpha'$, we insert the smallest value of $\Lambda(Y) \setminus \Lambda(Y_\alpha)$. By the convexity of geodesic length functions, there are at most $2$ possible isometry classes of $Y$ with this length of $\alpha'$. 

Now $\chi(Y_\alpha) | \leq \chi(X) = 2g-2$ and the lemma follows.
\end{proof}

We will need to argue differently if $\alpha'$ is of type (II). 

\subsubsection{Type (II) next shortest curves}

The projection of $\alpha'$ to $Y_\alpha$ consists of two arcs, freely homotopic to two simple orthogeodesics $a_1,a_2$ with the endpoints of $a_1$ on $\alpha_1$ and the endpoints of $a_2$ on $\alpha_2$. As previously, given $a_1$ and $a_2$ there are two possibilities for the homotopy class of $\alpha'$. (Here we have to be more careful: there are more cases to consider depending on the relative positions of the endpoints of $a_1$ and $a_2$.)

\begin{figure}[h]
{\color{linkred}
\leavevmode \SetLabels
\L(.168*1.025) $\alpha_1$\\%
\L(.405*1.025) $\alpha_2$\\%
\L(.574*1.025) $\alpha_1$\\%
\L(.815*1.025) $\alpha_2$\\%
\L(.167*.28) $\alpha'$\\
\L(.407*.28) $\alpha'$\\
\L(.573*.27) $a_1$\\%
\L(.816*.27) $a_2$\\%
\endSetLabels
\begin{center}
\AffixLabels{\centerline{\epsfig{file =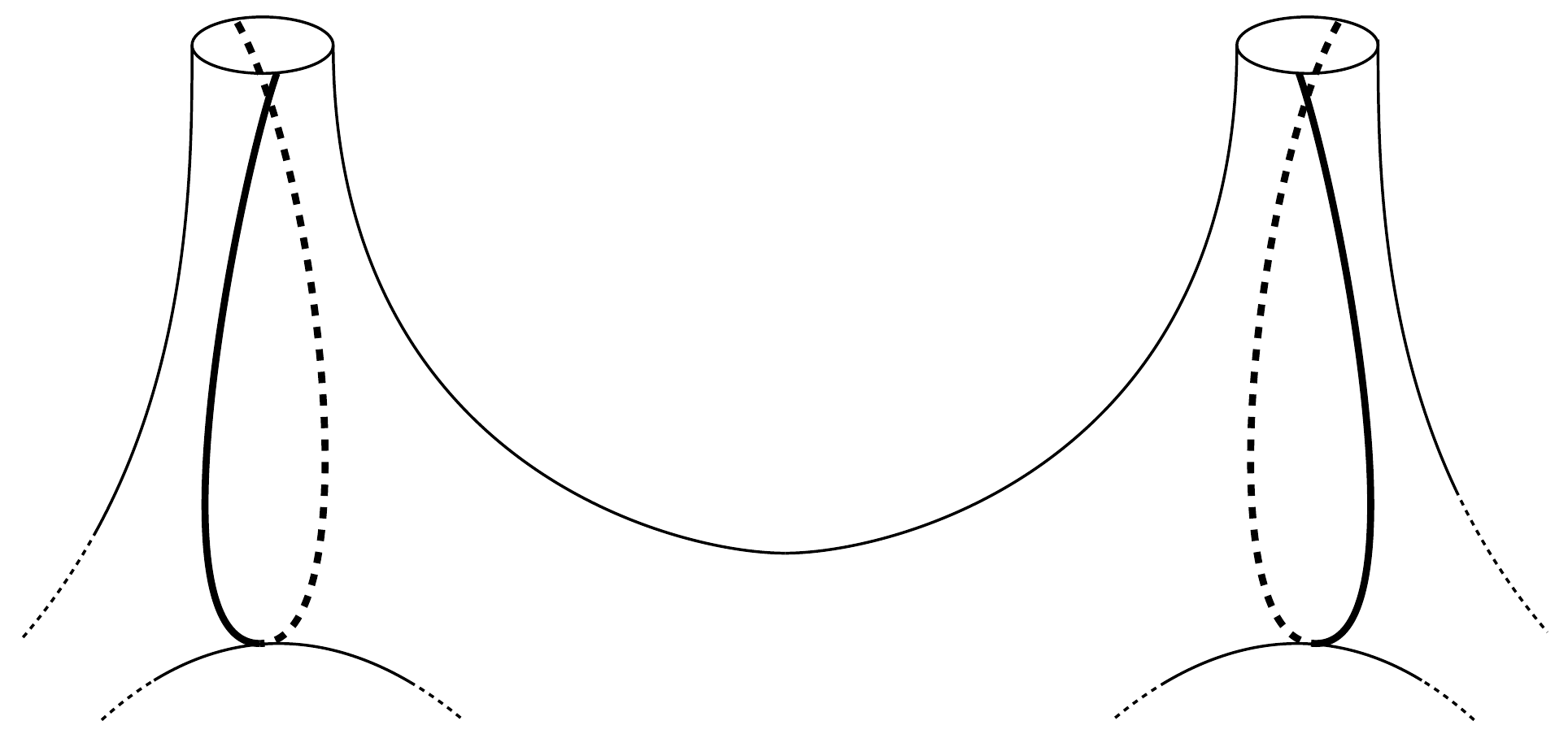,width=5.5cm,angle=0}\hspace{20pt}\epsfig{file =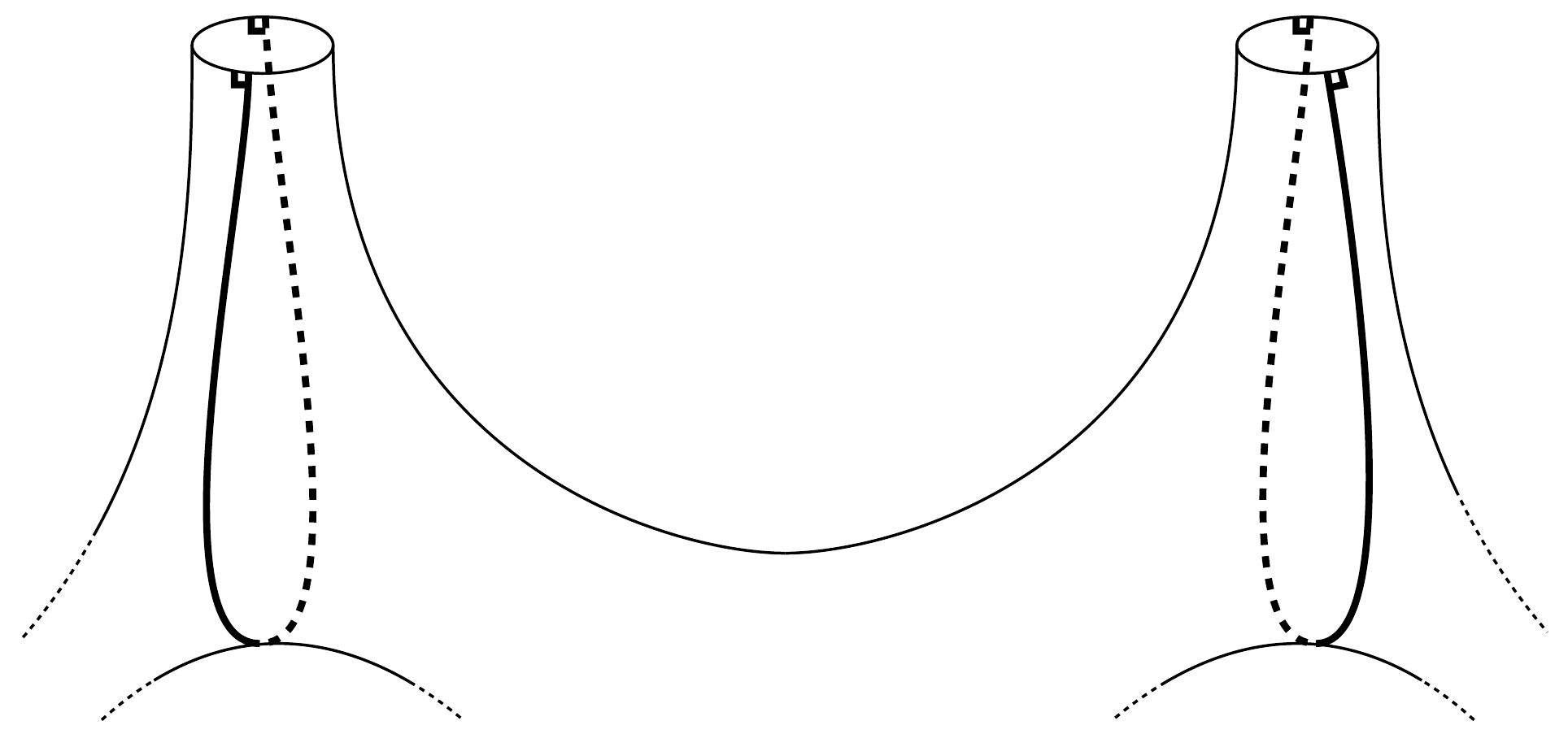,width=5.5cm,angle=0}}}
\vspace{-24pt}
\end{center}
}
\caption{When $\alpha'$ is of type (II)} \label{fig:typeIIarcs}
\end{figure}
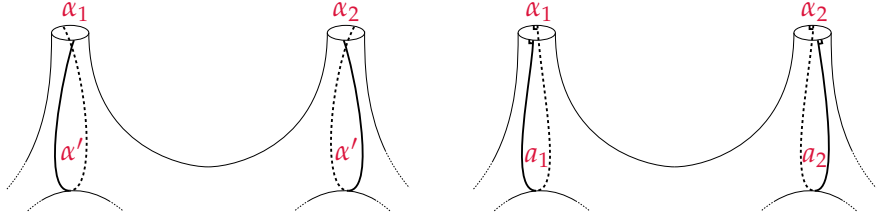

Note $Y_\alpha$, the arc $a_1$ lies in an embedded (geodesic) pair of pants $P_1$ with $\alpha_1$ as one of its boundary curves. Similarly, $a_2$ and $\alpha_2$ lies in an embedded (geodesic) pair of pants $P_2$ with $\alpha_2$ as one of its boundary curves. Using curve and chain systems, we can prove the following statement about the geometry of the pants $P_1$ and $P_2$.

\begin{lemma}\label{lem:nextgeom}
Let $\delta$ be a boundary curve of $P_1$ or $P_2$. Then
$$
\ell(\delta) < 6 \log(4g) + \arcsinh(1) + 2 \sqrt{2} \log(1+\sqrt{2}) < 6 \log(8 g)
$$
\end{lemma}

\begin{proof}
The lengths of the boundary curves of $P_1$ and $P_2$ depend on the lengths of $\alpha$ and the arcs $a_1$ and $a_2$. 

The arcs $a_i$ may not be the shortest orthogeodesics with endpoints on $\alpha_i$, but they can't be off by much. In fact, if they are off by more than $\frac{1}{2} \ell(\alpha)$, we can construct a curve that essential intersects $\alpha$ of length less than $\alpha'$. 

Let $a'_1$ and $a'_2$ the restrictions of $a_1$ and $a_2$ to $Y'_\alpha$. Note they are still orthogeodesics, even though they lie between non geodesic boundary. Further note that the restriction to $Y'_\alpha$ of any orthogeodesic of $Y_\alpha$ with both endpoints on $\alpha_1$ or $\alpha_2$ has its length reduced by exactly $2 w(\alpha)$ where $w(\alpha)$ is the width of $\CC(\alpha)$. This means that $a'_1$, resp. $a'_2$, is not more than $\frac{1}{2} \ell(\alpha)$ longer than the shortest orthogeodesics with endpoints both on $\alpha_1$, resp. $\alpha_2$.

In the proof of Theorem \ref{thm:cclength}, we used short orthogeodesics to find short curve and chain systems. Putting $\alpha_1$ and $\alpha_2$ in that context, we found orthogeodesics of length at most $2\log(4g)$ attached to them. If these orthogeodesics don't have both endpoints on $\alpha_1$, resp. $\alpha_2$, their second endpoint is on a curve of length at most $2 \log(4g)$. Using a concatenation of paths allows one to find an orthogeodesic of length at most $6 \log(4g)$ with both endpoints on $\alpha_1$, resp. $\alpha_2$. 

Thus for both $i=1,2$ we have $\ell(a_i) < 6 \log(4g) + \frac{1}{2} \ell(\alpha)$.

Now using these orthogeodesics, we can bound the lengths of the boundary curves of $P_1$ and $P_2$. For $i=1,2$, the boundary curves of $P_i$ are of length at most
$$
\ell(\alpha_i) + \ell(a_i)
$$
and as $\ell(\alpha) \leq 2 \arcsinh (1)$ and $\ell(\alpha_i) \leq 2 \sqrt{2} \log(1+\sqrt{2})$ by Corollary \ref{cor:shortboundary}, this proves the lemma. 
 \end{proof}

In order to use this bound on the length of boundary curves of the pants, we have the following lemma. The proof is based on a McShane type identity. It is of a somewhat different nature from the rest of the article, and possibly of independent interest, so we delay its proof to Appendix \ref{app:A}.

\begin{lemma}\label{lem:mcshanecount}
Let $Y$ be a surface with non-empty geodesic boundary and let $\beta$ be one of the boundary curves. Let $\mathcal{P}$ be the set of all embedded geodesic pairs of pants that have $\beta$ as a boundary curve and $\mathcal{P}_L$ the subset of $\mathcal{P}$ with the other two boundary curves of length at most $L$. Then the cardinality of $\mathcal{P}_L$ is at most $e^{L}$.
\end{lemma}
Note that the bound on the cardinality does not depend on either the topology or the length of $\beta$. 

We can now use the previous lemmas to prove the following. 

\begin{proposition}\label{prop:typeII}
There are at most $ 2 (8g)^{12}$ type (II) next shortest curves, hence at most $
4 (8g)^{12}$
isometry classes of $Y$ when $\alpha'$ is of type (II).
\end{proposition}

\begin{proof}
By Lemma \ref{lem:nextgeom}, the arcs $a_1$ and $a_2$ are found in embedded pants with boundary curves of length at most $L= 6 \log(8g)$. Now by Lemma \ref{lem:mcshanecount}, there are at most $e^L$ such embedded pants. Thus there are at most $e^{2L}$ choices for $a_1$ and $a_2$ and hence at most $ 2e^{2L}$ possible isotopy classes for $\alpha'$. By the convexity of length functions, this means there are at most $4 e^{2L}$ isometry types for $Y$. 
\end{proof}

\subsection{From thick to thin}

We now put this all together prove the following result, the main goal of this section.
\begin{theorem}\label{thm:thin}
Let $\Lambda(Y)$ be a length spectrum for some $Y\in \M_g$. Let $X_0$ be fixed isometry class of a thick part of a surface in $\M_g$ with $k_0$ curves in $\Gamma_0$. Then there are at most 
$$
8^{k_0+1} g^{12k_0}
$$
possible isometry types of $X$ with $X_0 \subset X$ and $\Lambda(X)=\Lambda(Y)$.
\end{theorem}
\begin{proof}
We proceed iteratively, curve by curve in $\Gamma_0$, adding one curve at a time and estimating the number of possible isometry types using Propositions \ref{prop:typeI} and \ref{prop:typeII}. We begin by knowing the isometry type of $X_0$, and then choose a curve $\alpha$ in $\Gamma_0$, and consider the surface $X_1 : = X \setminus \{ \Gamma_0 \setminus \alpha\}$. Depending on the type of $\alpha'$, we have bounds on the number of isometry classes for $X_1$. We then proceed to another curve in $\Gamma_0$ and so forth. Note that $X_{k_0}=X$. 

It is easy to check that our estimate for the number of isometry types for $\alpha'$ is of type (II) (Proposition \ref{prop:typeII}) exceeds the estimate when the curve is of type (I) (Proposition \ref{prop:typeI}). As at each step, $\alpha'$ is one or the other, a rough bound is thus at most twice the estimate of $8 \cdot (8g)^{12}$ on the isometry types proved in Proposition \ref{prop:typeII}.

Thus at the end of the process, we have at most $8^{k_0+1} g^{12k_0}$ possible isometry types for $X_{k_0}$ which proves the theorem.
\end{proof}

\section{Counting isometry types}\label{sec:iso}

We can now proceed to proving the full upper bound on the number of isospectral but non isometric surfaces in a given moduli space.

\begin{proof}[Proof of Theorem \ref{thm:maincount}]
Let $X$ be isospectral to $Y$. 

Let $\Gamma, \Gamma_\A$ be a short curve and chain system of $X$, the existence of which is guaranteed by Theorem \ref{thm:cclength}. As before, $\Gamma_1 := \Gamma \setminus \Gamma_0$ and $\Delta_{\Gamma_1}$ the set of curves transversal to those of $\gamma$ of length bounded in Lemma \ref{lem:trans}. Note $\Gamma, \Gamma_\A$ is one of $N_{cc}(g)$ different topological types. We'll take this into account at the end and for now concentrate on counting all possible isometry classes for $\Gamma, \Gamma_\A$ lying in a particular topological type. Recall the curves are marked.

By Lemma \ref{lem:countcurves}, there are at most $
(g-1) \, e^{L+6}$ different primitive closed geodesics of length less than $L$. 

Let $k$ be the number of curves of $\Gamma$. Note that $k\leq 3g-3$. In particular, by Theorem \ref{thm:cclength}, the $k$ lengths of the curves of $\Gamma$ are among a set of at most 
$$
(g-1) \, e^{2 \log(4g) +6} = (g-1)\, e^{\log(16e^6 g^2} = 16e^6 (g-1)g^2
$$
lengths. A (crude) upper is that there are at most $\left(16e^6 (g-1)g^2\right)^k$ choices for these curves.

Similarly, the lengths of the $6g-6$ curves of $\Gamma_\A$ are among a collection of
$$
(g-1) \, e^{8 \log(4g) +6} = 16e^6 (g-1)g^8
$$
lengths and thus there are at most $\left(16e^6 (g-1)g^8\right)^{6g-6}$ choices for these. 

These choices of lengths will determine the isometry type of $X\setminus \Gamma$ but we need more information to determine $X_0$. For this we'll use the curves $\Delta_{\Gamma_1}$. Using Lemma \ref{lem:trans}, we know there are at most
$$
(g-1) \, e^{14 \log(4g) +6} = 16e^6 (g-1)g^{14}
$$
choices for these curves. So if there are $k_1\leq k$ curves in $\Gamma_1$, this gives at most 
$$
\left(16e^6 (g-1)g^{14}\right)^{k_1}
$$
However, unlike before, the choice of length might not uniquely determine the isometry type. Fortunately, length is convex along twists. Thus for $\gamma\in \Gamma_1$ and $\delta_\gamma$  its transversal curve, the quantity $\ell(\delta_\gamma)$ will determine $\tau_\gamma$ up to two possibilities. When counting possible isometry type, we must thus multiply by an additional factor of $2^{k_1}$.

All in all, this shows that there are at most 
$$
N_{cc}(g) \cdot \left(16e^6 (g-1)g^2\right)^k \cdot \left(16e^6 (g-1)g^8\right)^{6g-6} \cdot \left(16e^6 (g-1)g^{14}\right)^{k_1}\cdot 2^{k_1}
$$
possible isometry types for $X_0$. 

We now use the results on thin surfaces to conclude. By Theorem \ref{thm:thin}, given $X_0$, there are at most 
$$
8^{k_0+1} g^{12k_0}
$$
possible isometry types for $X$ where $k_0$ is the cardinality of $\Gamma_0$. That gives us a total of 
\begin{equation}\label{eqn:bigcount}
N_{cc}(g) \cdot \left(16e^6 (g-1)g^2\right)^k \cdot \left(16e^6 (g-1)g^8\right)^{6g-6} \cdot \left(16e^6 (g-1)g^{14}\right)^{k_1}\cdot 2^{k_1} \cdot 8^{k_0+1} g^{12k_0}
\end{equation}
of possible isometry types for $X$, with the condition that $k_0+k_1= k \leq 3g-3$.

Unfortunately, we have to get into some messy estimations using our previous estimates for $N_{cc}(g)$. As much as possible, choices will be made in terms of their simplicity and not in terms of optimality. 

Using Lemma \ref{lem:countca} and some obvious simplifications, the quantity \ref{eqn:bigcount} becomes
$$
\frac{8}{e^6} \left( \frac{12^6}{e^5}\right)^{g-1} (g-1)^{6g-6} \left(256 e^6 (g-1)g^2\right)^k \cdot \left(16e^6 (g-1)g^8\right)^{6g-6} \cdot \left(32e^6 (g-1)g^{14}\right)^{k_1}\cdot g^{12k_0}
$$
which further simplifies to
$$
\frac{8}{e^6} \left( \frac{\left(12\cdot16\cdot e^6\right)^6}{e^5}\right)^{g-1}   \left(256 e^6 \right)^k \left(32e^6 \right)^{k_1}  \cdot \left(g-1\right)^{12g-12 +k +k_1} g^{48g-48 + 2k + 14k_1 + 12 k_0}
$$ 
Now using $k_1+ k_0\leq k \leq 3g-3$ and further simplications the above expression is bounded above by an expression of the form
$$
\frac{8}{e^6} A^{g-1} g^{B(g-1)}
$$
where $A = 3^6 2^{75} e^{67} \mbox{ and } B= 114$.

The main point is that the above bound is of the type $g^{Cg}$ for some $C$. Asymptotically it is certainly dominated by $g^{115 g}$ for large enough $g$, but for a cleaner expression, a small manipulation shows that for all $g\geq 2$ it is bounded above by $g^{154 g}$ which proves the result.
\end{proof}

\section{Length spectra interrogations}\label{sec:question}

In this final section we'll show that it suffices to ask an unknown spectrum a quantifiable finite number of questions to determine it uniquely. In some sense, most of the hard work has already been done. 

We recall the setup: we have an unknown length spectrum $\Lambda$ of a genus $g$ surface which we want to determine by asking questions. 

An {\it admissible question} is the following: what is the first value of $\Lambda \setminus \LL$ where $\LL$ is a finite list of values?

We want to determine the minimum number of questions we need to ask to determine $\Lambda$. 

The basic strategy is as follows. By the results on lengths of curve and chain systems, for some polynomial $p(g)$, it will suffice to know the first $p(g)$ lengths to determine all possible isometry types of the thick part of the underlying surface (the number of these being quantified). For each isometry type of thick surface, by asking an additional question for each short curve, we can determine all possible isometry types for the whole surface, again the number of these being quantified. Thus with a polynomial number of questions we can determine all possible isometry types, the number of these being at most $g^{Cg}$ for some $C$. We regroup isospectral isometry types and use additional questions to distinguish between them. 

We can now proceed to the proof of the main result. 

\begin{proof}[Proof of Theorem \ref{thm:mainquestion}]
To begin, we interrogate the beginning of the length spectrum $\Gamma$ to determine enough lengths to reconstruct the thick part of the surface up to computable finiteness. By the methods and proof of Theorem \ref{thm:maincount}, the lengths of curves of length at most $14 \log(4g)$ determine the thick part of a surface. Again by Lemma \ref{lem:trans}, these lengths are among the first
$$
(g-1) \, e^{14 \log(4g) +6} = 16e^6 (g-1)g^{14}
$$
lengths of $\Gamma$. Thus the first set of admissible questions is to ask for the first $16e^6 (g-1)g^{14}$ lengths. This gives us knowledge of the thick part of the surface up to computable finiteness, corresponding to all possible curve configurations.

Now to deal with the full surface, we proceed as in the proof of Theorem \ref{thm:thin}. Given a possible isometry type $X_0$ of the thick part of the surface, we look at $\Lambda \setminus \Lambda(X_0)$ and consider the first value in this set, say $\xi_1$. This can be determined via a single admissible question. Indeed, with the knowledge of $\Lambda(X_0)$, we can determine an upper bound on $\xi_1$, say $B_1$. Note that the position of $\xi_1$ in $\Lambda$ can be arbitrarily large if $\ell_1$ is arbitrarily small, but knowing $\ell_1 \in \Lambda(X_0)$ tells us where to look for $\xi_1$. The admissible question is: what is the first value of $\Lambda\setminus \{ x\in \Lambda(X_0) \mid x \leq B_1 \} ?$

With this is hand, we iterate this process as in Theorem \ref{thm:thin}: to determine the length of the $k$ curves traversal to those of $\Gamma_0$ requires $k$ admissible questions. So all in all, because $k\leq 3g-3$, we've asked at most $16e^6 (g-1)g^{14} +3g-3$ admissible questions. 

This finite set of questions has allowed us to determine all possible isometry classes of surfaces that might have $\Lambda$ as their length spectrum. Specifically, we have now reduced the problem to a collection of at most $I_g$ isometry classes where $I_g$ satisfies the bounds of Theorem \ref{thm:maincount}. Note this might seem somewhat counter-intuitive: the bounds from Theorem \ref{thm:maincount} are bounds on the number of isospectral but non-isometric surfaces whereas here we can't distinguish between isospectral surfaces. The key thing is that once we know an isometry class, we also know the length spectrum.

We can regroup these isometry classes by spectrum, as we won't be able to distinguish between non-isometric but isospectral surfaces by interrogating a length spectrum. Denote by $M$ this collection of possible spectra, each represented by an isometry class of surface with the appropriate length spectrum.

Now we need to ask additional admissible questions to figure out which one of these possible spectra $\Lambda$ really is. 

To do so, we take any two $X,Y$ with $\Lambda(X)\neq\Lambda(Y)$. There is a smallest integer $m_{X,Y}$ such that 
$$
\ell_{m_{X,Y}}(X) \neq \ell_{m_{X,Y}}(Y)
$$
We ask the following admissible question: what is the first value of 
$$
\Lambda\setminus \{ x\in \Lambda(X) \mid x < \ell_{m_{X,Y}}(X) \}
$$
If the answer is not $\ell_{m_{X,Y}}(X) $, then $\Lambda(X) \neq \Lambda$. If the answer is not $\ell_{m_{X,Y}}(Y) $ then $\Lambda(Y) \neq \Lambda$. So a single question rules out either $X$ or $Y$ (or possibly both). Once a spectrum is ruled out, we discard it. After at most $M-1$ questions, a single spectrum remains and we have determined $\Lambda$ with absolute certainty.

All in all we've asked
$$
M-1 + 16e^6 (g-1)g^{14} +3g-3
$$
possible questions so using the estimates from the proof of Theorem \ref{thm:maincount} the number of questions is bounded by 
$$
\frac{8}{e^6} A^{g-1} g^{B(g-1)} +16e^6 (g-1)g^{14} +3g-3
$$
where $A = 3^6 2^{75} e^{67}$ and$ B= 114$. As in the proof of Theorem  \ref{thm:maincount}, this estimate is asymptotically bounded above by $g^{115g}$ and via a small manipulation can be shown to be bounded, for all $g \geq 2$, by
$$g^{154g}$$
as claimed.
\end{proof}
\appendix
\section{Appendix: McShane identities}\label{app:A}

The main goal of the appendix is to prove Lemma \ref{lem:mcshanecount} using a version of the McShane identity. The identity we shall use is the following, due to Mirzakhani \cite{Mirzakhani} and independently discovered by Tan, Wong and Zhang who also proved a version for surfaces with cone points \cite{TanWongZhang}. 

\begin{theorem}\label{thm:mcshane}
Let $Y$ be a surface with non-empty geodesic boundary and let $\beta$ be one of the boundary curves. Let $\mathcal{P}$ be the set of all embedded geodesic pairs of pants that have $\beta$ as boundary curve. Let $\mathcal{P'}$ be the subset of $\mathcal{P}$ with two boundary curves of $Y$ as its boundary curves. 

Then there exist explicit positive functions $\mu, \eta$ that depend only the geometry of $P$ such that
$$
\sum_{P \in \mathcal{P} \setminus \mathcal{P'}} \mu(P) +  \sum_{P \in \mathcal{P'}} \eta(P) = 1
$$
\end{theorem}
The functions $\mu$ and $\eta$, often called gap functions, are functions of the boundary lengths of $P$. They are real functions in three variables. By convention we set the first variable of $\mu$ to be the length of $\beta$. For $\eta$, the first variable is the length of $\beta$ and the second variable is the length of the second boundary curve of the corresponding pair of pants. 

The following proposition contains all the features about these functions $\mu$ and $\eta$ that we will need.

\begin{proposition}\label{prop:mcshane}
The gaps functions $\mu$ and $\eta$ enjoy the following properties:
\begin{enumerate}[(i)]
\item $\mu(x,y,z) < \eta(x,y,z)$
\item $\mu(x, y,z) > \frac{1}{e^{\frac{y+z}{2}}}$
\end{enumerate}
\end{proposition}

\begin{proof}
By \cite{Mirzakhani} and \cite{TanWongZhang}, the explicit formulae for the functions $\mu$ and $\eta$ are the following:
$$
\mu(x,y,z)= \frac{4}{x} \arctanh\left(\frac{\sinh\left(\frac{x}{2}\right)}{\cosh\left(\frac{x}{2}\right)+e^{\frac{y+z}{2}}}\right)
$$
and
$$
\eta(x,y,z)=1 - \frac{2}{x} \arctanh\left(\frac{\sinh\left(\frac{x}{2}\right)\sinh\left(\frac{y}{2}\right)}{\cosh\left(\frac{z}{2}\right)+\cosh\left(\frac{x}{2}\right)\cosh\left(\frac{y}{2}\right)}\right)
$$
Both statements of the proposition can be shown by function manipulation, but we'll give a geometric reason for why (i) is true. The proof of (ii) however will be a straightforward function manipulation.
 
 We begin with (i) which, as we shall see, comes from analyzing where the functions in these identities come from.

These identities come from breaking up $\beta$ into intervals as follows. For each point of $\beta$, consider the unique geodesic segment that leaves $\beta$ from this point at a right angle and exponentiate until the corresponding geodesic either hits itself or hits a boundary geodesic. (The points that don't do either are measure $0$.) If it hits itself, return to $\beta$ by following the geodesic back to obtain a simple arc. Then we regroup the basepoints into segments according to the homotopy type of the associated simple arc. The functions (called gap functions) come from the measures of the segments divided by $\ell(\beta)$ for normalization. Each homotopy class of arc determines an embedded pair of pants.

The functions $\mu$ and $\eta$ are computed in the corresponding pair of pants. The intervals of $\beta$ (before normalization) are different depending on whether the pair of pants contains one or two boundary curves (these are the cases  $\mathcal{P} \setminus \mathcal{P'}$ and $ \mathcal{P'}$ above). The boundary points of the intervals correspond to simple geodesics that spiral indefinitely around the boundary curves of the pair of pants. For $\mathcal{P} \setminus \mathcal{P'}$, these spiraling geodesics are illustrated in \cite[Figure 1, p. 91]{TanWongZhang}. We reproduce a similar figure for convenience (Figure \ref{fig:PminusP'}).

\begin{figure}[h]
{\color{linkred}
\leavevmode \SetLabels
\L(.578*-.06) $\beta$\\%
\endSetLabels
\begin{center}
\AffixLabels{\centerline{\epsfig{file =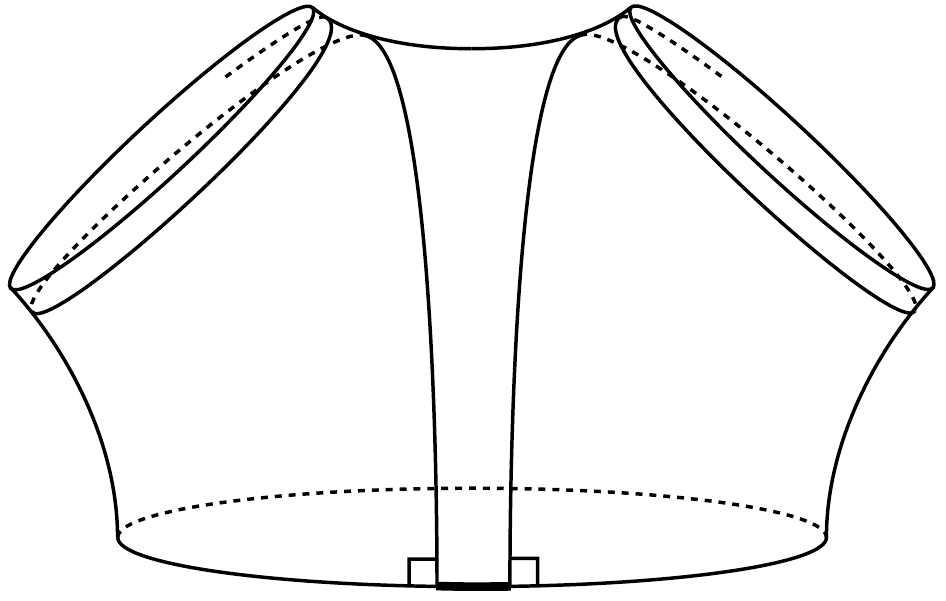,width=5.0cm,angle=0}}}
\vspace{-24pt}
\end{center}
}
\caption{The "front" interval on $\beta$ corresponding to this pair of pants is in bold} \label{fig:PminusP'}
\end{figure}

Note there are two of them, one for each orientation of the corresponding simple arc. A pair of pants in $\mathcal{P'}$ has a second boundary cuff $\beta'$. There is a single interval this time, but it encompasses both intervals that would have appeared had the corresponding pair of pants been in the set $\mathcal{P} \setminus \mathcal{P'}$. This is illustrated in Figure \ref{fig:P'}.

\begin{figure}[h]
{\color{linkred}
\leavevmode \SetLabels
\L(.578*-.06) $\beta$\\%
\L(.338*.645) $\beta'$\\%
\endSetLabels
\begin{center}
\AffixLabels{\centerline{\epsfig{file =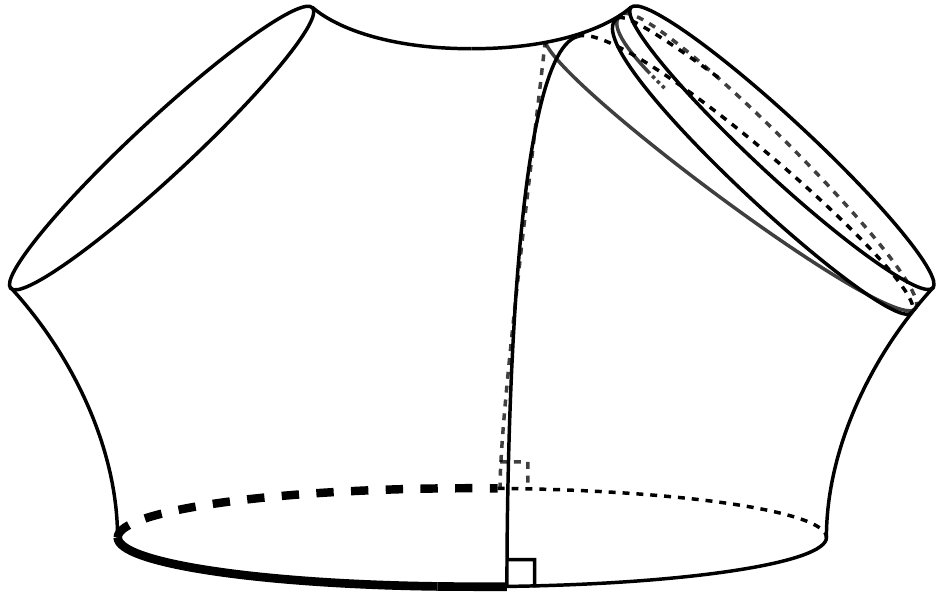,width=5.0cm,angle=0}}}
\vspace{-24pt}
\end{center}
}
\caption{The whole interval on $\beta$ corresponding to this pair of pants is in bold} \label{fig:P'}
\end{figure}

This, without doing a single computation, shows that $x \cdot \mu(x,y,z) < x \cdot \eta(x,y,z)$, and hence that $ \mu(x,y,z) < \eta(x,y,z)$.

We now pass to (ii). 

Using the definition of $\arctanh$ and setting $A:=e^{\frac{y+z}{2}}\geq 1$ we obtain
$$\mu(x,y,z)=
\frac{2}{x} \log\left( \frac{e^{\frac{x}{2}}+A}{e^{-\frac{x}{2}}+A}\right)
$$
We want to show $\mu(x,y,z) > \frac{1}{A}$ and to do this we'll show that 
$$
 \log\left( \frac{e^{\frac{x}{2}}+A}{e^{-\frac{x}{2}}+A}\right) > \frac{x}{2A}
$$
or equivalently that
$$
\frac{e^{\frac{x}{2}}+A}{e^{-\frac{x}{2}}+A} > e^{\frac{x}{2A}}
$$
In turn this is equivalent to showing that $F(x)>0$ for all $x>0$ (and $A\geq 1$) where
$$
F(x):=e^{\frac{x}{2}}+A - A e^{\frac{x}{2A}}-e^{\frac{x}{2}\left(\frac{1}{A}-1\right)}
$$
Note that $F(0) = 0$. We compute the derivative 
$$
F'(x) = \frac{1}{2} \left( e^{\frac{x}{2}}- e^{\frac{x}{2A}}\right) + \frac{1}{2}\left(1- \frac{1}{A}\right) e^{\frac{x}{2}\left(\frac{1}{A}-1\right)}
$$
which is positive for $A\geq 1$ and the claim follows.
\end{proof}

As a corollary, we can deduce Lemma \ref{lem:mcshanecount}. We recall that the statement is that cardinality of the set $\mathcal{P}_L$ is at most $e^{L}$. 

\begin{proof}[Proof of Lemma \ref{lem:mcshanecount}]
From Theorem \ref{thm:mcshane} and Proposition \ref{prop:mcshane} (i), we have
$$
\sum_{P \in \mathcal{P} \setminus \mathcal{P'}} \mu(P) +  \sum_{P \in \mathcal{P'}} \mu(P) = 1
$$
and thus
$$
\sum_{P \in \mathcal{P}_L} \mu(P) < 1
$$
Now by Proposition \ref{prop:mcshane} (ii), for $P$ with boundary lengths $\ell(\beta), y, z$ we have that $\mu(P) > \frac{1}{e^{\frac{y+z}{2}}}$. By the monotonicity of this lower bound we have
$$
\mu(P) > \frac{1}{e^L}
$$
for all $P\in \mathcal{P}_L$. The cardinality of $\mathcal{P}_L$ can thus be at most $e^L$.
\end{proof}

\addcontentsline{toc}{section}{References}
\bibliographystyle{plain}

\begin{thebibliography}{10}

\bibitem{Bavard}
Christophe Bavard.
\newblock Disques extr\'emaux et surfaces modulaires.
\newblock {\em Ann. Fac. Sci. Toulouse Math. (6)}, 5(2):191--202, 1996.

\bibitem{Bers}
Lipman Bers.
\newblock An inequality for {R}iemann surfaces.
\newblock In {\em Differential geometry and complex analysis}, pages 87--93.
  Springer, Berlin, 1985.

\bibitem{BhagwatRajan}
Chandrasheel Bhagwat and C.~S. Rajan.
\newblock On a multiplicity one property for the length spectra of even
  dimensional compact hyperbolic spaces.
\newblock {\em J. Number Theory}, 131(11):2239--2244, 2011.

\bibitem{BrooksGornetGustafson}
Robert Brooks, Ruth Gornet, and William~H. Gustafson.
\newblock Mutually isospectral {R}iemann surfaces.
\newblock {\em Adv. Math.}, 138(2):306--322, 1998.

\bibitem{BuserBook}
Peter Buser.
\newblock {\em Geometry and spectra of compact {R}iemann surfaces}, volume 106
  of {\em Progress in Mathematics}.
\newblock Birkh\"auser Boston, Inc., Boston, MA, 1992.

\bibitem{GordonWebbWolpert}
C.~Gordon, D.~Webb, and S.~Wolpert.
\newblock Isospectral plane domains and surfaces via {R}iemannian orbifolds.
\newblock {\em Invent. Math.}, 110(1):1--22, 1992.

\bibitem{Huber}
Heinz Huber.
\newblock Zur analytischen {T}heorie hyperbolischen {R}aumformen und
  {B}ewegungsgruppen.
\newblock {\em Math. Ann.}, 138:1--26, 1959.

\bibitem{Kac}
Mark Kac.
\newblock Can one hear the shape of a drum?
\newblock {\em Amer. Math. Monthly}, 73(4, part II):1--23, 1966.

\bibitem{Keen}
Linda Keen.
\newblock Collars on {R}iemann surfaces.
\newblock In {\em Discontinuous groups and {R}iemann surfaces ({P}roc. {C}onf.,
  {U}niv. {M}aryland, {C}ollege {P}ark, {M}d., 1973)}, pages 263--268. Ann. of
  Math. Studies, No. 79. Princeton Univ. Press, Princeton, N.J., 1974.

\bibitem{Margulis}
G.~A. Margulis.
\newblock Certain applications of ergodic theory to the investigation of
  manifolds of negative curvature.
\newblock {\em Funkcional. Anal. i Prilo\v zen.}, 3(4):89--90, 1969.

\bibitem{McKean}
H.~P. McKean.
\newblock Selberg's trace formula as applied to a compact {R}iemann surface.
\newblock {\em Comm. Pure Appl. Math.}, 25:225--246, 1972.

\bibitem{McShane}
Greg McShane.
\newblock Simple geodesics and a series constant over {T}eichmuller space.
\newblock {\em Invent. Math.}, 132(3):607--632, 1998.

\bibitem{Mirzakhani}
Maryam Mirzakhani.
\newblock Simple geodesics and {W}eil-{P}etersson volumes of moduli spaces of
  bordered {R}iemann surfaces.
\newblock {\em Invent. Math.}, 167(1):179--222, 2007.

\bibitem{MirzakhaniGrowth}
Maryam Mirzakhani.
\newblock Growth of the number of simple closed geodesics on hyperbolic
  surfaces.
\newblock {\em Ann. of Math. (2)}, 168(1):97--125, 2008.

\bibitem{ParlierBers}
Hugo Parlier.
\newblock A short note on short pants.
\newblock {\em Canad. Math. Bull.}, 57(4):870--876, 2014.

\bibitem{Przytycki}
Piotr Przytycki.
\newblock Arcs intersecting at most once.
\newblock {\em Geom. Funct. Anal.}, 25(2):658--670, 2015.

\bibitem{Sunada}
Toshikazu Sunada.
\newblock Riemannian coverings and isospectral manifolds.
\newblock {\em Ann. of Math. (2)}, 121(1):169--186, 1985.

\bibitem{TanWongZhang}
Ser~Peow Tan, Yan~Loi Wong, and Ying Zhang.
\newblock Generalizations of {M}c{S}hane's identity to hyperbolic
  cone-surfaces.
\newblock {\em J. Differential Geom.}, 72(1):73--112, 2006.

\bibitem{Tse}
Richard~M. Tse.
\newblock A lower bound for the number of isospectral surfaces.
\newblock In {\em Recent developments in geometry ({L}os {A}ngeles, {CA},
  1987)}, volume 101 of {\em Contemp. Math.}, pages 161--164. Amer. Math. Soc.,
  Providence, RI, 1989.

\bibitem{Vigneras1}
Marie-France Vign{\'e}ras.
\newblock Exemples de sous-groupes discrets non conjugu\'es de {${\rm
  PSL}(2,{\bf R})$} qui ont m\^eme fonction z\'eta de {S}elberg.
\newblock {\em C. R. Acad. Sci. Paris S\'er. A-B}, 287(2):A47--A49, 1978.

\bibitem{Vigneras2}
Marie-France Vign{\'e}ras.
\newblock Vari\'et\'es riemanniennes isospectrales et non isom\'etriques.
\newblock {\em Ann. of Math. (2)}, 112(1):21--32, 1980.

\bibitem{Wolpert}
Scott Wolpert.
\newblock The length spectra as moduli for compact {R}iemann surfaces.
\newblock {\em Ann. of Math. (2)}, 109(2):323--351, 1979.

\end{thebibliography}

\end{document}